\documentclass[a4paper,english]{article}
\usepackage{lmodern}
\usepackage{helvet}

\usepackage[T1]{fontenc}
\usepackage[latin9]{inputenc}
\usepackage{amsthm}
\usepackage{amsmath}
\usepackage{amssymb}
\usepackage{wasysym}
\usepackage{graphicx}
\usepackage{epstopdf}
\usepackage{setspace}
\usepackage{esint}
\usepackage[numbers]{natbib}

\makeatletter

\providecommand{\tabularnewline}{\\}

\theoremstyle{plain}
\newtheorem{thm}{\protect\theoremname}
  \theoremstyle{plain}
  \newtheorem{assumption}[thm]{\protect\assumptionname}
  \theoremstyle{definition}
  \newtheorem{example}[thm]{\protect\examplename}
 \newcommand\thmsname{\protect\theoremname}
 \newcommand\nm@thmtype{theorem}
 \theoremstyle{plain}
 
 \newenvironment{namedthm}[1][Undefined Theorem Name]{
   \ifx{#1}{Undefined Theorem Name}\renewcommand\nm@thmtype{theorem*}
   \else\renewcommand\thmsname{#1}\renewcommand\nm@thmtype{namedtheorem}
   \fi
   \begin{\nm@thmtype}}
   {\end{\nm@thmtype}}
  \theoremstyle{plain}
  \newtheorem{lem}[thm]{\protect\lemmaname}
  \theoremstyle{remark}
  \newtheorem{rem}[thm]{\protect\remarkname}
  \theoremstyle{plain}
  \newtheorem{prop}[thm]{\protect\propositionname}
  \theoremstyle{plain}
  \newtheorem{cor}[thm]{\protect\corollaryname}

\@ifundefined{date}{}{\date{}}
\usepackage{a4wide,diagbox}

\renewcommand{\hat}{\widehat}
\renewcommand{\tilde}{\widetilde}

\allowdisplaybreaks

\makeatother

\usepackage{babel}
  \providecommand{\assumptionname}{Assumption}
  \providecommand{\corollaryname}{Corollary}
  \providecommand{\examplename}{Example}
  \providecommand{\lemmaname}{Lemma}
  \providecommand{\propositionname}{Proposition}
  \providecommand{\remarkname}{Remark}
  \providecommand{\theoremname}{Theorem}
\providecommand{\theoremname}{Theorem}

\begin{document}
\global\long\def\E{\mathbb{E}_{\sigma,b,\gamma}}

\global\long\def\Var{\operatorname{Var}_{\sigma,b,\gamma}}

\global\long\def\I{\mathbf{1}}

\global\long\def\N{\mathbb{N}}

\global\long\def\R{\mathbb{R}}

\global\long\def\Z{\mathbb{Z}}

\global\long\def\F{\mathbb{\mathcal{F}}}

\global\long\def\L{\mathbb{\mathcal{L}}}

\global\long\def\P{\mathbb{P}_{\sigma,b,\gamma}}

\global\long\def\S{\mathcal{S}}

\global\long\def\T{\mathcal{T}}

\global\long\def\U{\mathcal{U}}

\global\long\def\epsilon{\varepsilon}

\global\long\def\MT{\clubsuit}

\global\long\def\JC{\spadesuit}

\title{Spectral estimation for diffusions with random sampling times}

\author{Jakub Chorowski and Mathias Trabs%
\thanks{The authors thank Markus Reiß for helpful comments and discussions.
J.C. was financially supported by the Deutsche Forschungsgemeinschaft
(DFG) RTG 1845 \textquotedblright Stochastic Analysis with Applications
in Biology, Finance and Physics\textquotedblright . M.T. acknowledges
financial support by the DFG through the CRC 649 ``Economic Risk''
and the research fellowship TR 1349/1-1. The main part of the paper
was carried out while M.T. was employed at the Humboldt-Universität
zu Berlin.%
}}

\date{Humboldt-Universität zu Berlin and Université Paris-Dauphine}
\maketitle
\begin{abstract}
The nonparametric estimation of the volatility and the drift coefficient
of a scalar diffusion is studied when the process is observed at random
time points. The constructed estimator generalizes the spectral method
by Gobet, Hoffmann and Reiß {[}Ann. Statist. 32 (2006), 2223-2253{]}.
The estimation procedure is optimal in the minimax sense and adaptive
with respect to the sampling time distribution and the regularity
of the coefficients. The proofs are based on the eigenvalue problem
for the generalized transition operator. The finite sample performance
is illustrated in a numerical example.
\end{abstract}
\begin{doublespace}
\textbf{MSC2010 subject classification:} Primary 62M05; Secondary
60J60, 62G99, 62M15. 
\end{doublespace}


\section{Introduction}

For decades diffusion models are used to describe the dynamics of
continuous stochastic processes, for instance, stock prices in econometrics
or particle movements in biology and physics. The statistical properties
of diffusion models depend essentially on the observation scheme,
where it is natural to assume discrete observations of the process.
Mostly, equidistant observations are studied in the literature, distinguishing
between high-frequent and low-frequent observations, depending whether
the observation distance tends to zero or remains fixed. A summary
of parametric methods is given by \citet{aitsahalia2010}. Nonparametric
estimation methods are surveyed by \citet{fan2005}. 

As argued by \citet{ait-sahaliaMykland:2003}, assuming equidistant
observations might however not be realistic in many applications and
random sampling times should be instead considered. For parametric
estimation problems \citet{ait-sahaliaMykland:2003,ait-sahaliaMykland:2004}
have shown that random sampling has a strong effect on the statistical
problem and the performance of estimators. Naturally, the question
arises how nonparametric estimators can be constructed for random
sampling times and whether their (asymptotic) behavior is similar
or worse than for equidistant observations.

In order to study the nonparametric estimation of the drift and the
volatility coefficient of the diffusion when the process is observed
at random times, we generalize the low-frequency results by \citet{GobetHoffmannReiss:2004}.
As they do, we consider a reflected scalar diffusion on a one-dimensional
interval. On the one hand, this allows to avoid technical difficulties
and to present more transparent proofs when investigating spectral
properties of the transition semigroup. On the other hand, diffusions
with reflecting barriers have rich applications. In the finance and
economics literature reflected diffusions are used for currency exchange
rate target-zone models, in which the exchange rate is allowed to
float within two barriers enforced by the monetary authority c.f.
\citep{BallRoma:1998,Krugman:1991,Svensson:1990}. Reflected diffusions
also appear as the payoff of the so-called ``Russian Options'',
c.f. \citet{SheppShiryaev:1993}. Among applications in mathematical
biology, we recall models for population dynamics in which the total
number of individuals is affected by oppositely acting forces, e.g.,
spontaneous growth and immigration on the one hand and random harvesting
or predation on the other, c.f. \citep{Ricciardi:1986}. Finally reflected
Brownian motion have been shown to describe queueing models experiencing
heavy traffic, see \citep{IglehartWhitt:1970,Kingman:1962}. In all
these models the observation times might not be equidistantly distributed.
For instance, they depend on trading times for finance applications
or measurement times of the biologist. 

By the compactness of the interval and the reflecting boundary, the
diffusion is ergodic and admits a spectral gap. Our procedure relies
on a representation of the coefficients in terms of the invariant
measure and the first non-trivial eigenpair of the infinitesimal generator
of the diffusion. This spectral identification method was introduced
in \citet{HansenScheinkmanTouzi1998:} and has been further studied
by \citep{ChenHansenScheinkman:2009}. It is crucial that the eigenpair
is determined by the transition operator of the time changed diffusion,
where the time change is given by the sampling distribution and the
Laplace transform of the sampling distribution. The former can be
estimated by a wavelet projection method and latter by classical empirical
process theory. As a side product of our analysis we clarify some
aspects of the estimator and the proofs by \citet{GobetHoffmannReiss:2004}.
In particular, in order to stabilize the estimator against large stochastic
errors a truncation with an in practice unknown threshold value was
needed, which we could omit. 

Moreover, we show that Lepski's method can be applied to chose the
projection level in a data-driven way. This allows to adapt on the
unknown Sobolev regularity of the drift and volatility coefficients
of the diffusion. The first adaptive estimator based on low-frequency
observations of a diffusion process has been constructed only recently
in \citet{SohlTrabs:2014}. Considering diffusion on the whole real
line, this first result is restricted to a diffusion with constant
volatility, simplifying the whole estimation problem, we do not need
any additional restrictions on the drift or the volatility. 

We prove that the estimators achieve minimax optimal convergence rates.
The adaptive estimator only loses a logarithmic factor. In view of
the \emph{cost of randomness} determined by \citet{ait-sahaliaMykland:2004},
it might be surprising that the convergence rates do not depend on
the sampling distribution and coincide in fact with the nonparametric
rates of the low-frequency setting. In that sense, our method is also
adaptive with respect to the unknown sampling distribution. As one
can see clearly from simulations, there is, however, a large \emph{cost
of ignoring the randomness} in the misspecified case where one applies
the low-frequency estimator to randomly sampled observations using
the average time step as observations distance.

The paper is organized as follows: In Section~\ref{sec:model} we
introduce the diffusion with reflected boundaries, our basic assumptions
and the main properties of the process. The estimators are constructed
in Section~\ref{sec:EstimationMethod}. The main results on the convergence
rates are stated and discussed in Section~\ref{sec:mainResults}.
The adaptive estimator is constructed in Section~\ref{sec:adaptive}.
The finite sample performance of the method is illustrated in a small
simulation study in Section~\ref{sec:sim}. The proofs of the upper
and lower bounds as well as for the Lepski method are postponed to
Sections~\ref{sec:Proofs}, \ref{sec:PrLowBound} and \ref{sec:proofAdaptiv},
respectively. Finally, some results on the stability of the eigenvalue
problems are presented in the appendix.

\section{The model\label{sec:model}}

Without loss of generality we can consider the unit interval $[0,1]$
for the reflecting diffusion. For a measurable and bounded drift function
$b\colon[0,1]\to\R$ and a continuous volatility function $\sigma\colon[0,1]\to\R_{+}$
let the process $X=\{X_{t}:t\ge0\}$ be given by the stochastic differential
equation
\begin{eqnarray}
dX_{t} & = & b\left(X_{t}\right)dt+\sigma\left(X_{t}\right)dW_{t}+v\left(X_{t}\right)dY_{t}\left(X\right),\label{eq:SDE}\\
X_{0} & = & x_{0},\,\text{and for all }t\geq0\ X_{t}\in[0,1],\nonumber 
\end{eqnarray}
where $x_{0}$ is a random variable on $[0,1]$, $W=\{W_{t}:t\geq0\}$
is a standard Brownian motion, $v\colon[0,1]\to\R$ satisfies $v(0)=1,v(1)=-1$,
and $Y$, which is part of the solution, is a non-anticipative continuous
non-decreasing process increasing only when $X_{t}\in\left\{ 0,1\right\} $.
By the Engelbert-Schmidt theorem boundedness of the drift coefficient
together with the volatility function being continuous and strictly
positive ensure that (\ref{eq:SDE}) has a weak solution, see \citet[Thm. 4.1]{RozkoszSlominski:1997}.
We denote by $\mathbb{P}_{\sigma,b}$ the law of this solution on
the canonical space $\Omega=C(\R_{+},[0,1])$ of continuous functions
equipped with the topology of uniform convergence on compact subsets
and endowed with its Borel $\sigma-$field $\F.$

For $N\in\N$ our observations are given by
\[
(0,X_{0}),(\tau_{1},X_{\tau_{1}}),\dots,(\tau_{N},X_{\tau_{N}})\in[0,\infty)\times[0,1]
\]
where $\tau_{1},\dots,\tau_{N}$ is an increasing sequence of random
time points. For convenience we write $\tau_{0}=0$.
\begin{assumption}
\label{ass:times}Let the observation distances
\[
\Delta_{n}:=\tau_{n}-\tau_{n-1},\quad n=1,\dots,N,
\]
be an independent and identically distributed sequence of strictly
positive random variables with law 
\[
\gamma\in\Gamma:=\Gamma(I,\alpha):=\big\{\gamma\text{ probability measure on }\R_{+}:\gamma(I)\ge\alpha\big\}
\]
for some compact interval $I\subset(0,\infty)$ and some $\alpha\in(0,1]$.
Let $\Delta_{n}$ be independent of the diffusion process $X$. 
\end{assumption}
This condition on the sampling distributions is very weak. For every
given positive distribution $\gamma$ there are $I,\alpha$ such that
$\gamma\in\Gamma(I,\alpha)$. The only restrictions are that the set
$\Gamma$ has to be bounded in the right sense, since we will derive
uniform rates in this class, and we have to exclude distributions
that concentrate at zero. The latter condition is natural because
otherwise the observations would be of high-frequency type which would
require a completely different analysis.
\begin{example}
\label{ex:samplingDist}~
\begin{enumerate}
\item The special case of the low-frequency observations is covered by setting
$\tau_{n}=n\Delta$ for some fixed deterministic $\Delta>0$. Then
the sampling distribution is given by the Dirac measure in $\Delta$,
that is $\Gamma=\{\delta_{\Delta}\}$.
\item If the observation times are governed by a Poisson process, the waiting
time to the next observation is exponentially distributed, that is
$\gamma=Exp(\lambda)$ for some intensity $\lambda>0$. In this case
we can choose $\Gamma=\{Exp(\lambda):\lambda\in\Lambda\}$ for any
bounded set $\Lambda\subset(0,\infty).$
\end{enumerate}
\end{example}
To state the assumptions on the diffusion coefficients, we denote
the $L^{2}([0,1])$ \emph{Sobolev space} of order $s>0$ by $H^{s}:=H^{s}([0,1])$.
Furthermore, let $H_{b}^{s}\subset H^{s}$ be the subset of bounded
functions with Sobolev regularity $s.$ Note that $H_{b}^{s}=H^{s}$
for $s>1/2$ by the Sobolev embeddings.
\begin{assumption}
\label{ass:drift volatility}For $s>1$ and constants $d,D>0$ let
$(\sigma,b)\in\Theta_{s}$ where
\[
\Theta_{s}:=\Theta_{s}(d,D)=\left\{ (\sigma,b)\in H^{s}\times H_{b}^{s-1}:\|\sigma^{2}\|_{H^{s}}\leq D,\,\|b\|_{H^{s-1}}\leq D,\inf_{x}\sigma(x)\geq d\right\} .
\]

\end{assumption}
In particular, $(\sigma,b)\in\Theta_{s}$ ensures the existence of
a weak solution of (\ref{eq:SDE}). As shown by \citet{GobetHoffmannReiss:2004}
the compactness of $[0,1]$ and the reflecting boundary conditions
imply that $X$ has a spectral gap and thus it is geometrically ergodic
and admits an invariant measure $\mu$. Focusing on asymptotic results,
we can suppose that the initial value $x_{0}$ is distributed according
to $\mu$. Assumption~\ref{ass:drift volatility} implies that $\mu$
has the Lebesgue density, abusing notation denoted by $\mu$ as well,
\begin{equation}
\mu(x):=\mu_{\sigma,b}(x)=C_{0}\sigma^{-2}(x)\exp\Big(\int_{0}^{x}2b(x)\sigma^{-2}(y)\, dy\Big),\quad x\in[0,1],\label{eq:mu}
\end{equation}
for some normalizing constant $C_{0}>0$, cf. \citet[Chap. 4]{Bass:1995}
or \citet[Chap. 15, Sect. 6]{KarlinTaylor1981:}. It is easy to see
that the regularity assumptions on $b$ and $\sigma$ imply that $\mu\in H^{s}$,
which will be essential for the analysis of the estimators. From the
explicit formula for $\mu$ moreover follows that there are constants
$0<c<C$ such that $c\le\mu_{\sigma,b}\le C$ for any $(\sigma,b)\in\Theta_{s}$.
Consequently, $L^{2}(\mu)$ with the inner product 
\[
\langle f,g\rangle_{\mu}:=\int_{0}^{1}f(x)g(x)\mu(x)dx
\]
 is a Hilbert space equivalent to $L^{2}\left([0,1]\right)$. 

Noting that reflection corresponds to Neumann boundary conditions,
the infinitesimal generator $L=L_{\sigma,b}$ of the diffusion $X$
is an unbounded, densely defined operator on $L^{2}([0,1])$ satisfying
\begin{eqnarray*}
Lf(x) & = & b(x)f'(x)+\frac{1}{2}\sigma^{2}(x)f''(x),\\
\text{dom}(L) & = & \left\{ f\in H^{2}([0,1]):f'(0)=f'(1)=0\right\} .
\end{eqnarray*}
Furthermore, seen as an operator on the Hilbert space $L^{2}(\mu)$,
the generator $L$ is an elliptic, self-adjoint operator with compact
resolvent, see \citet[Example 4.21]{Chatelin:1983}. Consequently
it has a pure point spectrum $\sigma(L)=\left\{ v_{k}:k=0,1,...\right\} $
and the corresponding eigenfunctions $u_{k}$ form an $L^{2}(\mu)$
orthogonal basis. Its largest eigenvalue $v_{0}$ equals $0$ with
constant corresponding eigenfunction. All other eigenvalues are negative
and we assume that they are ordered with respect to their multiplicities
$0>v_{1}\geq v_{2}\geq...$ . As shown in \citep[ Lemma 6.1]{GobetHoffmannReiss:2004},
the eigenvalue $v_{1}$ is simple and the eigenfunction $u_{1}$ can
be chosen strictly increasing.

\section{Estimation method\label{sec:EstimationMethod}}

\subsection{Spectral identification}

The main idea used for the construction of the spectral estimators
in \citep{GobetHoffmannReiss:2004} is that the coefficients of a
stationary diffusion process can be expressed in terms of the invariant
density $\mu$ and any nontrivial eigenpair $(v_{k},u_{k})$, $k\ge1$.
Indeed, expressing the invariant measure in terms of the speed measure
together with the Neumann boundary conditions yields, cf. \citep[Sect. 3.1]{GobetHoffmannReiss:2004},

\begin{align}
\sigma^{2}(x) & =\frac{2v_{k}\int_{0}^{x}u_{k}(y)\mu(y)dy}{u_{k}'(x)\mu(x)},\label{eq:sigma}\\
b(x) & =\frac{v_{k}u_{k}(x)}{u_{k}'(x)}-\frac{\sigma^{2}(x)u_{k}''(x)}{2u_{k}'(x)}\label{eq:b}\\
 & =v_{k}\frac{u_{k}(x)u'_{k}(x)\mu(x)-u_{k}''(x)\int_{0}^{x}u_{k}(y)\mu(y)dy}{u_{k}'(x)^{2}\mu(x)}.\nonumber 
\end{align}
Applying the ergodicity, it is easy to estimate the invariant measure
$\mu$. To recover an eigenpair of the generator, \citet{GobetHoffmannReiss:2004}
have used discrete equidistant observations, i.e. $\Delta_{n}=\Delta$
for some fixed $\Delta>0$, to construct a matrix estimator of the
transition operator $P_{\Delta}=e^{\Delta L}$. Noting that $P_{\Delta}$
shares eigenfunctions with the generator $L$ while its eigenvalues
are $e^{\Delta v_{k}}$, $k=0,1,...$, they have obtained estimators
of $(v_{k},u_{k})$. We will generalize these results taking into
account the random observation times $\tau_{1},\dots,\tau_{N}$. 

Similar to the transition operator $P_{\Delta}$ we introduce the
\emph{generalized transition operator} $R$ on $L^{2}(\mu)$ given
by

\begin{equation}
Rf(x)=\E\left[f(X_{\tau})|X_{0}=x\right],\quad x\in[0,1],\label{eq:R}
\end{equation}
where $\tau$ is a random variable with distribution $\gamma$ being
independent of the process $X$. The crucial insight is that for any
eigenpair $(v_{k},u_{k})$ of the generator we have 
\begin{equation}
Ru_{k}(x)=\E\big[\E[P_{t}u_{k}|\tau=t]\big]=\mathbb{E}_{\gamma}\left[e^{\tau v_{k}}\right]u_{k}(x)=\underset{=:\kappa_{k}}{\underbrace{\L_{\gamma}(-v_{k})}}\cdot u_{k}(x),\label{eq:kappa}
\end{equation}
where
\begin{equation}
\mathcal{L}_{\gamma}(z):=\int_{0}^{\infty}e^{-tz}\gamma(dt),\quad z\in\R_{+},\label{eq:Laplace}
\end{equation}
is the Laplace transform of $\gamma.$ Consequently, $R$ is a compact
operator with eigenvalues $1=\kappa_{0}>\kappa_{1}>\kappa_{2}\geq\kappa_{3}\geq...>0$.
In the functional calculus sense we obtain 
\[
R=\L_{\gamma}\left(-L\right).
\]

Therefore, we can estimate the eigenpairs $(v_{k},u_{k})$ using the
spectral properties of $R$. Since the sampling distribution $\gamma$
is unknown, we need to estimate the Laplace transform from the observations
$(\Delta_{n})_{n=1,\dots,N}$.
\begin{namedthm}[Example~\ref{ex:samplingDist} (continued)]

\begin{enumerate}
\item For $\Delta_{n}\equiv\Delta$ for some fixed $\Delta>0$ we have $Rf=P_{\Delta}f$
and $\mathcal{L}_{\gamma}(z)=e^{-\Delta z},z\ge0$. We thus exactly
recover the situation studied in \citep{GobetHoffmannReiss:2004}.
\item If $\Delta_{n}\sim Exp(\lambda)$, then the Laplace transform is given
by $\mathcal{L}_{\gamma}(z)=\int_{0}^{\infty}\lambda e^{-t(z+\lambda)}dt=\frac{\lambda}{z+\lambda},z\ge0$
and the operator $R$ is the resolvent of the generator $L$. 
\end{enumerate}
\end{namedthm}
The distribution of the eigenvalues of the operator $R$ is inherited
from the generator $L$ and the sampling distribution $\gamma$. More
precisely, we obtain the following lemma whose proof is postponed
to Section~\ref{sub:Spectral properties of R}.
\begin{lem}
\label{lem:eigenvaluesR}Grant Assumptions~\ref{ass:times} and \ref{ass:drift volatility}.
The spectral gap, that is $\inf_{i\neq1}|\kappa_{i}-\kappa_{1}|$,
and the eigenvalues of the generalized transition operator $R$ have
a lower bound uniform in $(\sigma,b)\in\Theta_{s}$ and $\gamma\in\Gamma$.
\end{lem}

\subsection{Construction of the estimators}

Let us fix some notation. We will write $f\lesssim g$ (resp. $g\gtrsim f$)
when $f\leq C\cdot g$ for some universal constant $C>0$. $f\sim g$
is equivalent to $f\lesssim g$ and $g\lesssim f$. Let $(\psi_{\lambda})$,
with multi-indices $\lambda=(j,k)$, be an $L^{2}-$orthonormal regular
wavelet basis of $L^{2}([0,1]).$ The corresponding approximation
spaces are given by 
\[
V_{J}:=\overline{\text{span}}\big\{\psi_{\lambda}:|\lambda|=\left|(j,k)\right|:=j\leq J\big\}.
\]
The $L^{2}-$orthogonal and the $L^{2}(\mu)-$orthogonal projections
onto $V_{J}$ are denoted by $\pi_{J}$ and $\pi_{J}^{\mu}$, respectively. 

In fact, the approximation spaces do not necessarily need to be generated
by wavelets. We only require that $V_{J},J\in\N,$ satisfy Jackson
and Bernstein type inequalities with respect to the Sobolev spaces
$H^{s}$, that is for all $0\leq t\leq s,f\in H^{s}$ and $g\in V_{J}$
\begin{align}
\|(I-\pi_{J})f\|_{H^{t}} & \lesssim2^{-J(s-t)}\|f\|_{H^{s}}\quad\text{and}\quad\|g\|_{H^{j}}\lesssim2^{Jj}\|g\|_{L^{2}},\quad j=1,2,\label{eq:bernsteinJackson}
\end{align}
and additionally we need the uniform bound
\begin{equation}
\Big\|\sum_{|\lambda|\leq J}\psi_{\lambda}^{2}\Big\|_{\infty}\lesssim\text{dim}(V_{J})=2^{J}.\label{eq:Sum of squared base functions}
\end{equation}
It follows from the well known properties of wavelets that (\ref{eq:bernsteinJackson})
and (\ref{eq:Sum of squared base functions}) are satisfied.
\begin{rem}
\label{rem:basis}Since the eigenfunctions of the generator of the
reflected Brownian motion are given by the trigonometric functions,
it seams to be attractive to choose $V_{J}$ as the closure of the
span of the first $2^{J}$ orthogonal trigonometric basis functions,
which however does not fulfill (\ref{eq:bernsteinJackson}). If the
drift and the volatility function satisfy the stronger Hölder regularity
assumption $\|\sigma^{2}\|_{C^{s}},\|b\|_{C^{s-1}}\leq D,$ where
$\|\cdot\|_{C^{s}}$ denotes the Hölder norm, then we can obtain the
same bounds on the mean $L^{2}$ estimation error under a weaker version
of Jackson's inequality, namely
\[
\|(I-\pi_{J})f\|_{L^{2}}\lesssim2^{-Js}\|f\|_{C^{s}}.
\]
This inequality is satisfied for the trigonometric basis. Furthermore
Bernstein's inequality can be easily checked and (\ref{eq:Sum of squared base functions})
is trivially fulfilled. The same applies to the B-spline basis, that
satisfies above conditions with the weakened Jackson inequality (see
\citep{Ciesielski:1963} and \citep{CiesielskiFigiel:1982}).
\end{rem}
After having fixed the basis functions and the corresponding approximation
spaces $V_{J}$, there is a one-to-one correspondence between a linear
operator $A\colon V_{J}\to V_{J}$ on the finite dimensional space
$V_{J}$ and its matrix representation $(A_{\lambda,\lambda'})\in\R^{\dim V_{J}\times\dim V_{J}}$
with $A_{\lambda,\lambda'}:=\langle\psi_{\lambda},A\psi_{\lambda'}\rangle$.
To simplify the notation, we will throughout use $A$ to denote the
operator as well as its representation matrix.

Using the ergodicity of the diffusion $X$ and the independence of
$X$ and $(\Delta_{n})_{n}$, the sequence $(X_{\tau_{n}})_{n}$ is
ergodic, too. The natural estimator for the invariant measure is therefore
the empirical measure 
\[
\mu_{N}=\frac{1}{N+1}\sum_{n=0}^{N}\delta_{X_{\tau_{n}}}.
\]
To regularize $\mu_{N}$, we define the projection estimator 
\begin{align*}
\hat{\mu}_{J}(x) & :=\sum_{|\lambda|\leq J}\langle\psi_{\lambda},\mu_{N}\rangle\psi_{\lambda}(x)\quad\text{with}\quad\langle\psi_{\lambda},\mu_{N}\rangle:=\frac{1}{N+1}\sum_{n=0}^{N}\psi_{\lambda}(X_{\tau_{n}})
\end{align*}
for a projection level $J\in\N$. We proceed similarly to \citet{GobetHoffmannReiss:2004}.
Extending the matrix estimator of the transition semigroup, we introduce
the matrix estimator $\hat{R}_{J}=(\hat{R}_{\lambda,\lambda'})$ of
the action of the operator $R$ from (\ref{eq:R}) on the wavelet
basis with respect to the scalar product $\langle\cdot,\cdot\rangle_{\mu}$:
\[
\hat{R}_{\lambda,\lambda'}:=\frac{1}{2N}\sum_{n=0}^{N-1}\Big(\psi_{\lambda}\left(X_{\tau_{n+1}}\right)\psi_{\lambda'}\left(X_{\tau_{n}}\right)+\psi_{\lambda'}\left(X_{\tau_{n+1}}\right)\psi_{\lambda}\left(X_{\tau_{n}}\right)\Big).
\]
Since the observation times are independent from the diffusion, conditioning
on $\tau_{n}$, we can verify that $\hat{R}_{J}$ is an unbiased estimator
of the action of the operator $R$ on the basis, that is

\[
\E\big[\hat{R}_{\lambda,\lambda'}\big]=\langle\psi_{\lambda},R\psi_{\lambda'}\rangle_{\mu}.
\]
The Gram matrix $G_{J}=(\langle\psi_{\lambda},\psi_{\lambda'}\rangle_{\mu})_{\lambda,\lambda'}\in\R^{\dim V_{J}\times\dim V_{J}}$
is determined by $\langle v,G_{J}v\rangle=\langle v,v\rangle_{\mu}$
for all $v\in V_{J}\setminus\left\{ 0\right\} $. Hence, $G_{J}$
is a restriction of the scalar product $\langle\cdot,\cdot\rangle_{\mu}$
to finite dimensional space $V_{J}$. It can be estimated by $\hat{G}_{J}=(\hat{G}_{\lambda,\lambda'})$
with
\begin{align*}
\hat{G}_{\lambda,\lambda'} & =\frac{1}{N}\Big(\frac{1}{2}\psi_{\lambda}\big(X_{0}\big)\psi_{\lambda'}\big(X_{0}\big)+\sum_{n=1}^{N-1}\psi_{\lambda}\big(X_{\tau_{n}}\big)\psi_{\lambda'}\big(X_{\tau_{n}}\big)+\frac{1}{2}\psi_{\lambda}\big(X_{\tau_{N}}\big)\psi_{\lambda'}\big(X_{\tau_{N}}\big)\Big),
\end{align*}
satisfying

\[
\E\big[\hat{G}_{\lambda,\lambda'}\big]=\langle\psi_{\lambda},\psi_{\lambda'}\rangle_{\mu}=\langle\psi_{\lambda},G_{J}\psi_{\lambda'}\rangle.
\]
Owing to $\langle v,G_{J}v\rangle=\langle v,v\rangle_{\mu}>0$ for
any $v\in V_{J}\setminus\left\{ 0\right\} $, the matrix $G_{J}$
is invertible. By construction $\langle v,\hat{G}_{J}v\rangle$ is
always non-negative and it will be strictly positive whenever the
sample is sufficiently dispersed over all the interval $[0,1]$. By
ergodicity we can expect this to be a high probability event. With
a Neumann series argument we can moreover bound the norm of $\hat{G}_{J}^{-1}$
as stated by the following lemma, which is proven in Section~ \ref{sub:Analysis stoch error}.
\begin{lem}
\label{lem:Set T_0 size}Grant Assumption~\ref{ass:times} and \ref{ass:drift volatility}.
On the event $\T_{1}=\Big\{\|G_{J}-\hat{G}_{J}\|_{L^{2}}\leq\frac{1}{2}\|G_{J}^{-1}\|_{L^{2}}^{-1}\Big\}$
the estimator $\hat{G}_{J}$ is invertible and satisfies $\big\|\hat{G}_{J}{}^{-1}\big\|_{L^{2}}\leq2\|G_{J}^{-1}\|_{L^{2}}$.
Moreover, $\P\left(\Omega\setminus\T_{1}\right)\leq N^{-1}2^{2J}$
holds uniformly over $\Theta_{s}$ and $\Gamma$. 
\end{lem}
Whenever $\hat{G}_{J}^{-1}$ exists, we can consider $\hat{G}_{J}^{-1}\hat{R}_{J}$.
Since $\hat{R}_{J}$ is symmetric it immediately follows that $\hat{G}_{J}^{-1}\hat{R}_{J}$
is symmetric with respect to the $\hat{G}_{J}$-scalar product. Furthermore,
by the Cauchy-Schwarz inequality and the inequality between geometric
and arithmetic means we obtain for all $v\in V_{J}\setminus\left\{ 0\right\} $
\begin{eqnarray*}
\langle\hat{R}_{J}v,v\rangle & = & \frac{1}{N}\sum_{n=0}^{N-1}v\left(X_{\tau_{n}}\right)v\left(X_{\tau_{n+1}}\right)\\
 & \leq & \frac{1}{N}\Big(\sum_{n=0}^{N-1}v^{2}\big(X_{\tau_{n}}\big)\Big)^{1/2}\Big(\sum_{n=1}^{N}v^{2}\big(X_{\tau_{n}}\big)\Big)^{1/2}\\
 & \leq & \frac{1}{N}\Big(\frac{1}{2}v^{2}\left(X_{0}\right)+\frac{1}{2}v^{2}\left(X_{\tau_{N}}\right)+\sum_{n=1}^{N-1}v^{2}\left(X_{\tau_{n}}\right)\Big)=\langle\hat{G}_{J}v,v\rangle.
\end{eqnarray*}
Consequently, all eigenvalues of the matrix $\hat{G}_{J}^{-1}\hat{R}_{J}$
are real and smaller than one. It is easy to check that $1$ is an
eigenvalue corresponding to the constant function. We define the estimator
$\left(\hat{\kappa}_{J,1},\hat{u}_{J,1}\right)$ of the eigenpair
$\left(\kappa_{1},u_{1}\right)$ as the eigenpair of the matrix $\hat{G}_{J}^{-1}\hat{R}_{J}$
corresponding to the biggest eigenvalue smaller than one. On the exceptional
event that $\hat{G}_{J}$ is not invertible, we set $\hat{\kappa}_{J,1}=0$
and $\hat{u}_{J,1}=1$. Furthermore we choose the estimated eigenfunction
$\hat{u}_{J,1}$ normalized in $L^{2}$. 

Using $\hat{\kappa}_{J,1}$ and the identification equation $\kappa_{1}=\mathcal{L}_{\gamma}(-v_{1})$,
we can estimate $v_{1}$. The canonical estimator for the Laplace
transform of $\gamma$ is the Laplace transform of the empirical measure
of the sampling distances $\Delta_{n}=\tau_{n}-\tau_{n-1},n=1,\dots,N$.
Hence, we define
\[
\hat{\mathcal{L}}(y):=\frac{1}{N}\sum_{n=1}^{N}e^{-y\Delta_{n}},\quad y\in\R_{+}.
\]
Due to the i.i.d. structure of $(\Delta_{n})$, the classical empirical
process theory shows that $\hat{\mathcal{L}}$ estimates $\mathcal{L}_{\gamma}$
uniformly in a neighborhood of $v_{1}$ with the parametric rate $N^{-1/2}$.
Moreover, $\hat{\mathcal{L}}$ is strictly decreasing and continuous,
thus invertible. We define
\begin{equation}
\hat{v}_{J,1}:=-\hat{\L}^{-1}\left(\hat{\kappa}_{J,1}\right)\mathbf{1}_{\{\hat{\kappa}_{J,1}>0\}}.\label{eq:vHat}
\end{equation}

With the above definitions and in view of the identification formulas
(\ref{eq:sigma}) and (\ref{eq:b}) we can define the plug-in estimators
of the diffusion coefficients. In order to ensure integrability of
our estimators, we need to stabilize against large stochastic errors.
Using the prior knowledge that $(\sigma,b)\in\Theta_{s}$, especially
$\|\sigma^{2}\|_{\infty}\le D$ and $\|b\|_{L^{2}}\le D$ for some
$D>0$, we thus define 
\begin{gather}
\hat{\sigma}_{J}^{2}(x)=2\hat{v}_{J,1}\frac{\int_{0}^{x}\hat{u}_{J,1}(y)\hat{\mu}_{J}(y)dy}{\hat{u}_{J,1}'(x)\hat{\mu}_{J}(x)}\wedge D,\label{eq:Sigma est definition}\\
\hat{b}_{J}(x)=\tilde{b}_{J}(x)\I_{\{\|\tilde{b}_{J}\|_{L^{2}}\leq2D\}}\quad\text{for}\quad\tilde{b}_{J}(x):=\frac{\hat{v}_{J,1}\hat{u}_{J,1}(x)}{\hat{u}_{J,1}'(x)}-\frac{\hat{\sigma}_{J}^{2}(x)\hat{u}_{J,1}''(x)}{2\hat{u}_{J,1}'(x)}.\label{eq:drift est definition}
\end{gather}

\section{Minimax convergence rates\label{sec:mainResults}}

Let us now state our first main results, generalizing Theorems 2.4
and 2.5 in \citep{GobetHoffmannReiss:2004}, respectively. Note that
since $u_{1}'(0)=u_{1}'(1)=0$ the function 
\[
[0,1]\ni x\mapsto\frac{2v_{1}\int_{0}^{x}u_{1}(y)\mu(y)dy}{u'_{1}(x)\mu(x)}=\sigma^{2}(x)
\]
is defined in $\{0,1\}$ via continuous extension such that the proposed
estimators $\hat{\sigma}_{J}^{2}$ and $\hat{b}_{J}$ might be unstable
at the boundary. We restrict the $L^{2}$-loss to an interval $[a,b]\subset[0,1]$
for $0<a<b<1$ and refer to \citep[Section 3.3.8]{GobetHoffmannReiss:2004}
for a discussion of the boundary problem. 
\begin{thm}
\label{thm:Upper bounds}Grant Assumptions~\ref{ass:times} and \ref{ass:drift volatility}
for some $s>1$. Let $0<a<b<1$. Choosing $2^{J}\sim N^{1/(2s+3)}$,
we have 
\begin{eqnarray*}
\sup_{(\sigma^{2},b,\gamma)\in\Theta_{s}\times\Gamma}\E\big[\|\hat{\sigma}_{J}^{2}-\sigma^{2}\|_{L^{2}([a,b])}^{2}\big] & \lesssim & N^{-2s/(2s+3)},\\
\sup_{(\sigma^{2},b,\gamma)\in\Theta_{s}\times\Gamma}\E\big[\|\hat{b}_{J}-b\|_{L^{2}([a,b])}^{2}\big] & \lesssim & N^{-2(s-1)/(2s+3)}.
\end{eqnarray*}

\end{thm}
The risk of $\hat{\sigma}^{2}$ and $\hat{b}$ decomposes into the
errors for estimating the invariant density $\mu$ and the eigenpair
and $(v_{1},u_{1})$ of the infinitesimal generator $L$ of the diffusion.
In view of formula (\ref{eq:mu}) the invariant density inherits Sobolev
regularity of degree $s$ from the diffusion coefficients. Together
with the ergodicity and the spectral gap $\mu$ can be estimated with
the rate $\E[\|\hat{\mu}_{J}-\mu\|_{L^{2}}]\lesssim N^{-\frac{s}{2s+1}}$
if we choose $2^{J}\sim N^{-1/(2s+1)}$, cf. Proposition~\ref{prop:Invariant measure bound}.
Due to $\mathcal{L}_{\gamma}(-v_{1})=\kappa_{1}$ estimating $v_{1}$
reduces to estimate the eigenvalue $\kappa_{1}$ of the operator $R$
and the inverse of the Laplace transform $\mathcal{L}_{\gamma}$ in
a neighborhood of $\kappa_{1}$. The latter estimation problem can
be solved with standard empirical process results yielding the parametric
rate $N^{-1/2}$ for $\hat{\text{\ensuremath{\mathcal{L}}}}$, see
Lemma~\ref{lem:vHat}. 

The analysis of the estimation error of the eigenpair $(\kappa_{1},u_{1})$
of the generalized transition operator $R$ is the most challenging
ingredient of our proofs. We first restrict the eigenvalue problem
to the finite dimensional space $V_{J}$, that is we find $(\kappa_{J,1},u_{J,1})\in\R_{+}\times V_{J}$
such that 
\begin{equation}
\langle v,Ru_{J,1}\rangle_{\mu}=\kappa_{J,1}\langle v,u_{J,1}\rangle_{\mu}\quad\text{ for all }v\in V_{J}.\label{eq:Proof sketch finite EP}
\end{equation}
As shown in Theorem~\ref{thm:EP for self-adjoint compact positive-definite operator}
the resulting approximation error $\|u_{1}-u_{J,1}\|_{L^{2}(\mu)}+|\kappa_{1}-\kappa_{J,1}|$
is controlled by the spectral gap of the operator $R$ and the smoothness
of the eigenfunction (of degree $s+1$) achieving the order of magnitude
$2^{-J(s+1)}$. In the second step we approximate the finite dimensional
problem (\ref{eq:Proof sketch finite EP}) by a generalized symmetric
eigenvalue problem for the random matrices $\hat{R}_{J}$ and $\hat{G}_{J}$.
We use classical a posteriori error bounds to show that the approximation
error is controlled by the norm of the so called residual vector $r=(\hat{R}_{J}-\kappa_{J,1}\hat{G}_{J})u_{J,1}$,
cf. Theorem~\ref{thm:GHEP}. $\|r\|_{L^{2}}$ can be bounded by the
matrix approximation errors $\|(\hat{R}_{J}-R_{J})u_{J,1}\|_{L^{2}}$
and $\|(\hat{G}_{J}-G_{J})u_{J,1}\|_{L^{2}}$ that tend to zero by
the mixing property of the Markov chain $(X_{\tau_{n}})_{n}$. A delicate
point is that the a posteriori technique gives an existence statement,
but does not bound the error between ordered eigenpairs. We overcome
this difficulty using the absolute Weyl theorem for generalized symmetric
eigenvalue problems, see \citep{Nakatsukasa:2010}. We conclude that
$(\kappa_{1},u_{1})$ can be estimated with the rate $N^{-(s+1)/(2s+3)}$.

Because the volatility estimator relies on the first derivative of
the eigenfunction the statistical problem is ill-posed of degree one,
deteriorating the rate to $N^{-s/(2s+3)}$. For the drift estimator
we need the second derivative, adding a degree of ill-posedness. At
the same time the regularity of $b$ is smaller such that the rate
becomes $N^{-(s-1)/(2s+3)}=N^{-(s-1)/(2(s-1)+5)}$. Compared to \citet{GobetHoffmannReiss:2004},
the same rates can thus be achieved with random sampling times (with
unknown sampling distribution) than with equidistant low frequent
observations. In fact, the convergence rates are optimal in the minmax
sense: 
\begin{thm}
\label{thm:Lower bounds}Grant Assumption~\ref{ass:times} for an
arbitrary $\gamma\in\Gamma$ admitting a bounded Lebesgue density
at the origin. Grant Assumption~\ref{ass:drift volatility} for some
$s>1$. For $0<a<b<1$ it holds
\begin{eqnarray*}
\inf_{\bar{\sigma}}\sup_{(\sigma^{2},b)\in\Theta_{s}}\E\big[\|\bar{\sigma}^{2}-\sigma^{2}\|_{L^{2}([a,b])}^{2}\big] & \gtrsim & N^{-2s/(2s+3)},\\
\inf_{\bar{b}}\sup_{(\sigma^{2},b)\in\Theta_{s}}\E\big[\|\bar{b}-b\|_{L^{2}([a,b])}^{2}\big] & \gtrsim & N^{-2(s-1)/(2s+3)},
\end{eqnarray*}
where the infimum is taken over all estimators, i.e. measurable functions,
$\bar{\sigma}$ and $\bar{b}$, respectively.
\end{thm}
The proof of the lower bounds for observations sampled at random times
follows the same strategy as for low frequency observations in \citep{GobetHoffmannReiss:2004}.
Constructing alternatives that admit the same invariant measure, proving
the lower bound is reduced to a testing problem by Assouad's lemma,
see \citet[Sect. 2.7.2]{tsybakov:2009}. The Kullback-Leibler distance
between the distributions of two alternatives can then be bounded
in terms of the $L^{2}-$distance between the kernels of the corresponding
operators $R$ from (\ref{eq:R}), which is finally accomplished using
Hilbert-Schmidt norm estimates and the explicit form of the inverse
of the generator.

\section{Adaptive estimation\label{sec:adaptive}}

The optimal choice of the projection level crucially depends on the
unknown smoothness $s$. In this section, we construct a completely
data driven estimation procedure adapting to the Sobolev regularity
of $\sigma^{2}$ and $b$. We focus on the volatility estimator, noting
that the methodology should extend to the drift estimation without
additional theoretical problems. We adopt the general adaption principle
by \citet{Lepski:1990}. 

The aim is to chose the optimal projection level from the set
\[
\mathcal{J}_{N}:=[J_{min},J_{max}]\cap\N\quad\mbox{with}\quad2^{J_{min}}\sim\log N,\quad2^{J_{max}}\sim{\textstyle \frac{N}{(\log N)^{2}\log\log N}}
\]
For any $J\in\mathcal{J}_{N}$ we define 
\begin{equation}
s_{J}^{2}:=\Lambda^{2}2^{3J}\frac{\log\log N}{N}\label{eq:radius}
\end{equation}
for some appropriate constant $\Lambda>0$ depending on $d,D$ as
well as $I,\alpha$ (but not on $s$) from the Assumptions \ref{ass:times}
and \ref{ass:drift volatility}. The quantity $s_{J}$ is an upper
bound for the stochastic error of $\hat{\sigma}_{J}^{2}$, cf. Corollary~\ref{cor:concentration}.
The adaptive estimator is defined by
\[
\tilde{\sigma}^{2}:=\hat{\sigma}_{\hat{J}}^{2}\quad\mbox{with}\quad\hat{J}:=\min\big\{ J\in\mathcal{J}_{N}:\forall_{K\ge J,K\in\mathcal{J}_{N}}\|\hat{\sigma}_{K}^{2}-\hat{\sigma}_{J}^{2}\|_{L^{2}([a,b])}\le s_{K}\big\}.
\]
Heuristically, $\hat{J}$ is the smallest projection level for which
the stochastic error still dominates the bias. 

Our main result for the adaptive estimation shows that the estimator
$\tilde{\sigma}^{2}$ achieves the optimal convergence rate up to
an additional $\log\log N$ factor.
\begin{thm}
\label{thm:adaptiveEstimation}Grant Assumptions~\ref{ass:times}
and define $\Gamma_{0}:=\{\gamma\in\Gamma:\mathbb{E}_{\gamma}[\tau^{-1/2}]\le D\}$.
Let Assumption~\ref{ass:drift volatility} be fulfilled for some
$s>5/2$. Let $0<a<b<1$. Then there exists for every $\epsilon>0$
some $C>0$ such that, for $N$ sufficiently large, we have 
\begin{align*}
\sup_{\begin{array}[t]{c}
(\sigma,b,\gamma)\in\Theta_{s}\times\Gamma_{0}\end{array}}\P\Big(\|\tilde{\sigma}^{2}-\sigma^{2}\|_{L^{2}([a,b])}^{2}>C\Big(\frac{\log\log N}{N}\Big)^{2s/(2s+3)}\Big) & <\epsilon.
\end{align*}

\end{thm}
The proof of this theorem is postponed to Section~\ref{sec:proofAdaptiv}.
It relies on a concentration inequality for the Markov chain $(X_{\tau_{n}})_{n\ge0}$,
see Proposition~\ref{prop:concentration} as well as \citet[Section 3]{NicklSohl:2015}.
For the latter we need the additional assumption on $\gamma$ allowing
for a uniform bound on the transition density of the time-changed
diffusion process. Up to the concentration result, the proof relies
on the standard arguments for the Lepski method.

\section{Numerical example\label{sec:sim}}

In this section, we present numerical results for the volatility estimation.
Throughout the chapter, we consider a diffusion process $X$ with
linear mean reverting drift $b(x)=0.2-0.4x$, quadratic squared volatility
function $\sigma^{2}(x)=0.4-(x-0.5)^{2}$ and two reflecting barriers
at 0 and 1. The sample paths were generated using Euler-Maruyama scheme
with time step size 0.001 and reflection after each step.

For $\Delta=0.25$ we compare the estimation error for four different
sampling distributions of quite different shapes: the case of equidistant
observations with frequency $\Delta^{-1}$, the uniform distribution
on the interval $[0,2\Delta]$, the symmetric Beta$(0.2,0.2)$ distribution
rescaled to the interval $[0,2\Delta]$ and finally, the exponential
distribution with intensity $\Delta^{-1}$. Note that all considered
distributions have mean $\Delta$, Uniform and Beta distribution have
the same compact support $[0,2\Delta]$ and together with exponential
distribution they allow for arbitrary small sampling distances. Figure
\ref{fig:Sample path fragment} depicts a fragment of a simulated
trajectory of the diffusion together with the observations from different
sampling schemes. 
\begin{figure}[tp]
\centering{}\includegraphics[width=0.95\textwidth]{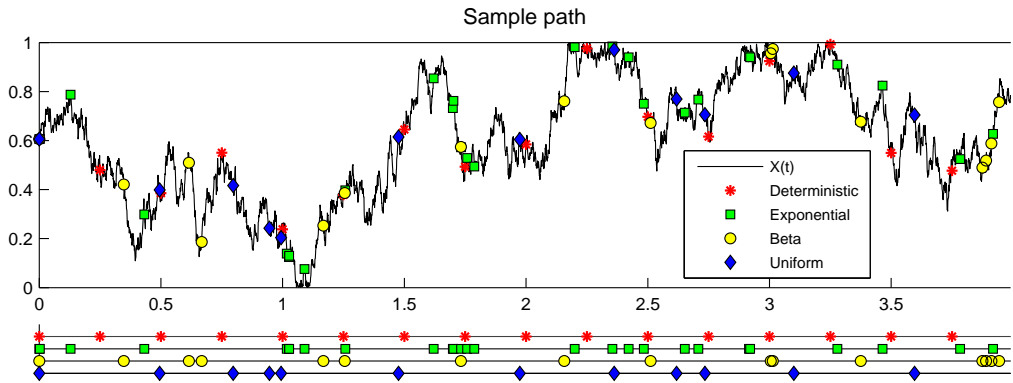}\protect\caption{\label{fig:Sample path fragment}Sample path of the process $X$ for
$0\leq t\leq4$ with marked observations from different sampling distributions.}
\end{figure}

To construct the approximation spaces, we used the Fourier orthogonal
cosines basis i.e.
\[
V_{J}=\overline{\operatorname{span}}\{\sqrt{2}\cos(j\pi x):0\leq j\leq J\},
\]
cf. Remark~\ref{rem:basis}. We compare an oracle choice of the projection
level with the adpative estimator. As target interval we choose $[0.1,0.9].$

\begin{table}[tp]
\begin{centering}
\begin{tabular}{|l|c|c|c|c|c|c|}
\cline{2-7} 
\multicolumn{1}{l|}{} & \multicolumn{3}{c|}{Oracle projection level} & \multicolumn{3}{c|}{Adaptive estimator}\tabularnewline
\hline 
\diagbox{Distribution}{Sample Size} & 4 000 & 12 000 &  20 000 & 4 000 & 12 000 & 20 000\tabularnewline
\hline 
\hline 
Deterministic & 0.0233 & 0.0155 & 0.0123 & 0.0318 & 0.0214 & 0.0130\tabularnewline
\hline 
Uniform & 0.0258 & 0.0168 & 0.0134 & 0.0341 & 0.0221 & 0.0139\tabularnewline
\hline 
Exponential & 0.0282 & 0.0177 & 0.0141 & 0.0362 & 0.0231 & 0.0148\tabularnewline
\hline 
Beta & 0.0296 & 0.0211 & 0.0179 & 0.0432 & 0.0255 & 0.0178\tabularnewline
\hline 
\end{tabular}
\par\end{centering}

\protect\caption{\label{tab:RMSE table} Root mean integrated squared error for volatility
estimation on $[0.1,0.9]$ based on 1000 Monte Carlo iterations. }
\end{table}
In Table \ref{tab:RMSE table} we compare the oracle and adaptive
root mean integrated squared error (RMISE) for volatility estimation
on the interval $[0.1,0.9],$ obtained by a Monte Carlo simulation
with 1000 iterations. The oracle projection level $J$ is stable with
respect to the sampling distribution and surprisingly small, taking
values 2 for N = 4~000 and 4 for N = 12~000 and N = 20~000 across
all distributions, with the exception of Beta with sample size N =
12~000, when it equals 2. For the adaptive estimation we chose the
constant $\Lambda$ in (\ref{eq:radius}) equal to 0.01. 

Relative to $\|\sigma^{2}\|_{L^{2}([0.1,0.9])}\approx0.31$ the error
of the oracle decreases from approximately 10\% for sample size N
= 4~000 to 5\% for N = 20~000. In particular for large sample sized
the error of the adaptive procedure is fairly close to the oracle
error. The errors are quite stable across sampling distributions as
the estimator, where the deterministic sampling allows for the smallest
error and the Beta distribution generates the largest errors. The
latter is not surprising because the Beta distribution is chosen in
a way that yields a strong clustering of the observations.

For 20 independent paths and sample size N = 20 000 the resulting
adaptive volatility estimators are shown in Figure \ref{fig:Volatility flows}.
While the estimators behave nicely in the interior of the interval,
the boundary problem outside the interval $[0.1,0.9]$ is clearly
visible. Again we see that the estimation for the Beta sampling distribution
is the worst.

\begin{figure}[tp]
\centering{}\includegraphics[width=0.95\columnwidth]{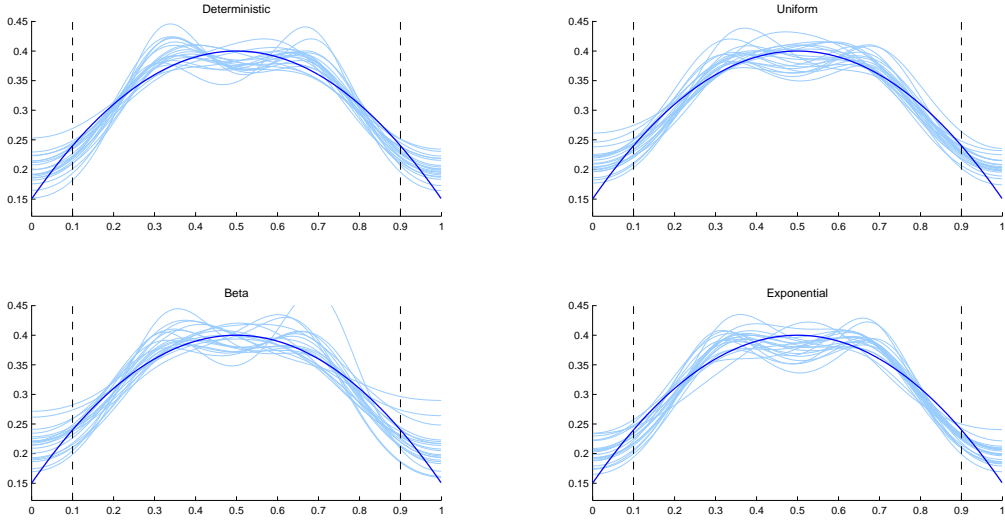}\protect\caption{\label{fig:Volatility flows}Estimated volatility functions using
adapted estimator for 20 independent trajectories of the diffusion
and four different sampling distributions with sample size N = 20
000.}
\end{figure}

In the misspecified case where the randomness of the observation times
is ignored, the RMISE of the low-frequency estimator designed for
equidistant observations with $\Delta$ set to the average observation
distance is four times larger than the error of our method in our
simulations.

\section{Proofs of the upper bounds\label{sec:Proofs}}

Throughout we take Assumptions~\ref{ass:times} and \ref{ass:drift volatility}
for granted.

\subsection{Spectral properties of the generalized transition operator $R$\label{sub:Spectral properties of R}}

Recall that $u_{1}$ is the eigenfunction corresponding to the biggest
negative eigenvalue $v_{1}$ of the generator $L$, normalized in
$L^{2}([0,1])$. By \citep[Proposition 6.5]{GobetHoffmannReiss:2004}
$u_{1}$ can be chosen to be increasing and for any $0<a<b<1$ there
exists a positive constant $c_{a,b}>0$ such that
\begin{equation}
\inf_{(\sigma,b)\in\Theta_{s}}\inf_{x\in[a,b]}u_{1}'(x)>c_{a,b}.\label{eq:lower bound for derivative of eigenfunction}
\end{equation}
By Lemma 6.1 in \citep{GobetHoffmannReiss:2004} the family of generators
$\left\{ L_{\sigma,b}:(\sigma,b)\in\Theta_{s}\right\} $ has a uniform
spectral gap on $\Theta_{s}$ meaning that there is a constant $s_{0}>0$
such that

\begin{equation}
\inf_{(\sigma,b)\in\Theta_{s}}\inf_{i\neq1}|v_{i}-v_{1}|=\inf_{(\sigma,b)\in\Theta_{s}}\left\{ \left|v_{1}\right|,\left|v_{2}-v_{1}\right|\right\} \geq s_{0}.\label{eq:spectral gap}
\end{equation}
Moreover the eigenvalues $v_{k}$ satisfy uniformly on $\Theta_{s}$
\begin{equation}
C_{1}k^{2}\leq-v_{k}\leq C_{2}k^{2},\label{eq:eigenvalue bounds}
\end{equation}
for constants $0<C_{1}<C_{2},$ while corresponding eigenfunctions
$u_{k}$ belong to the Sobolev space $H^{s+1}$ fulfilling

\begin{equation}
\|u_{k}\|_{H^{s+1}}\lesssim(1\vee|v_{k}|)^{\left\lceil s\right\rceil }.\label{eq:eigenfunctions bounds}
\end{equation}
As announced in Lemma (\ref{lem:eigenvaluesR}) these bounds transfer
uniformly to the operator $R$.
\begin{proof}[Proof of Lemma \ref{lem:eigenvaluesR}]
 For convenience we define $m:=\min I>0$ and $M:=\max I$. By the
definition of $R$ and the uniform bounds on the eigenvalues $v_{k}$
of $L$ in (\ref{eq:eigenvalue bounds}), we have
\[
\kappa_{k}=\L_{\gamma}(-v_{k})=\int_{0}^{\infty}e^{tv_{k}}\gamma(dt)\geq\int_{0}^{\infty}e^{-tC_{2}k^{2}}\gamma(dt)\geq\alpha e^{-MC_{2}k^{2}}\quad\text{for }k\geq1.
\]
The spectral gap of the operator $R$ equals $\min\left\{ 1-\kappa_{1},\kappa_{1}-\kappa_{2}\right\} $.
Due to (\ref{eq:spectral gap}), we have
\begin{eqnarray*}
\kappa_{1}-\kappa_{2} & = & \int_{0}^{\infty}\big(e^{tv_{1}}-e^{tv_{2}}\big)\gamma(dt)=\int_{0}^{\infty}e^{tv_{2}}\big(e^{t(v_{1}-v_{2})}-1\big)\gamma(dt)\\
 & \geq & \int_{0}^{\infty}e^{-4tC_{2}}\big(e^{ts_{0}}-1\big)\gamma(dt)\geq\alpha e^{-4MC_{2}}\big(e^{ms_{0}}-1\big).
\end{eqnarray*}
Similarly $1-\kappa_{1}=\int_{0}^{\infty}\big(1-e^{tv_{1}}\big)\gamma(dt)\geq\int_{0}^{\infty}\big(1-e^{-tC_{1}}\big)\gamma(dt)\geq\alpha\big(1-e^{-mC_{1}}\big).$
\end{proof}

\subsection{Consequences of the mixing property}

First we establish general bounds for the variance of integrals with
respect to the empirical measure which are due to the mixing behavior
of the sequence $(X_{\tau_{k}})_{k}.$ The following Lemma is a straightforward
generalization of \citep[Lemma 6.2]{GobetHoffmannReiss:2004}. Since
this is the key result to bound the stochastic error, we give the
proof to keep the paper self-contained. 
\begin{lem}
\label{lem:Variance estimates}For bounded $H_{1},H_{2}\in L^{2}([0,1])$
we have the following two variance estimates:
\begin{eqnarray*}
\Var\Big[\frac{1}{N}\sum_{n=1}^{N}H_{1}(X_{\tau_{n}})\Big] & \lesssim & N^{-1}\E\left[H_{1}^{2}(X_{0})\right],\\
\Var\Big[\frac{1}{N}\sum_{n=0}^{N-1}H_{1}(X_{\tau_{n}})H_{2}(X_{\tau_{n+1}})\Big] & \lesssim & N^{-1}\E\big[H_{1}^{2}(X_{0})H_{2}^{2}(X_{\tau_{1}})\big].
\end{eqnarray*}
\end{lem}
\begin{proof}
Denote $f\left(X_{\tau_{n}}\right)=H_{1}\left(X_{\tau_{n}}\right)-\E\left[H_{1}\left(X_{\tau_{n}}\right)\right]$.
Consider $m\geq n$ and let $k=m-n$. Since process $X$ is stationary
and has a uniform spectral gap $\|R^{k}f\|_{L^{2}(\mu)}\leq\|f\|_{L^{2}(\mu)}\L_{\gamma}^{k}(s_{0})$
holds for every function $f$ that is $L^{2}(\mu)$- orthogonal to
constants. Arguing analogously as in the proof of Lemma \ref{lem:eigenvaluesR}we
obtain $\sup_{\gamma\in\Gamma}\L_{\gamma}(s_{0})<1$. Hence, by the
Cauchy-Schwarz inequality,
\begin{eqnarray*}
\E\left[f\left(X_{\tau_{n}}\right)f\left(X_{\tau_{m}}\right)\right] & = & \E\left[f\left(X_{\tau_{n}}\right)\E\left[f\left(X_{\tau_{n+k}}\right)|X_{\tau_{n}}\right]\right]\\
 & = & \langle f,R^{k}f\rangle_{\mu}\leq\|f\|_{L^{2}(\mu)}^{2}\L_{\gamma}^{k}(s_{0}).
\end{eqnarray*}
Since $\|f\|_{L^{2}(\mu)}^{2}=\Var\left[H_{1}\left(X_{0}\right)\right]\leq\E\left[H_{1}^{2}\left(X_{0}\right)\right]$
and 
\[
\Var\Big[\sum_{n=1}^{N}H_{1}(X_{\tau_{n}})\Big]=\sum_{n,m=1}^{N}\E\left[f(X_{\tau_{n}})f(X_{\tau_{m}})\right]\le\|f\|_{L^{2}(\mu)}^{2}\sum_{n,m=1}^{N}\L_{\gamma}^{|n-m|}(s_{0})
\]
 to prove the first inequality we just have to show that $\sum_{n,m=1}^{N}\L_{\gamma}^{|n-m|}(s_{0})\lesssim N$.
This easily follows from the formula for the sum of finite geometric
series.

To prove the second inequality, first note that
\begin{alignat*}{1}
 & \Var\Big[\frac{1}{N}\sum_{n=0}^{N-1}H_{1}(X_{\tau_{n}})H_{2}(X_{\tau_{n+1}})\Big]\leq\frac{1}{N^{2}}\E\Big[\sum_{n,m=0}^{N-1}H_{1}(X_{\tau_{n}})H_{2}(X_{\tau_{n+1}})H_{1}(X_{\tau_{m}})H_{2}(X_{\tau_{m+1}})\Big].\\
 & \quad=\frac{1}{N^{2}}\E\Big[\sum_{n,m=0}^{N-1}H_{1}(X_{\tau_{n}})H_{2}(X_{\tau_{n+1}})H_{1}(X_{\tau_{m}})H_{2}(X_{\tau_{m+1}})\Big]-\langle H_{1},RH_{2}\rangle_{\mu}^{2}.
\end{alignat*}
Since the sum of diagonal terms equals $N^{-1}\E\Big[H_{1}^{2}(X_{0})H_{2}^{2}(X_{\tau_{1}})\Big]$,
it does not exceed the claimed upper bound. The sum of the other terms
equals 
\[
\frac{1}{N^{2}}\sum_{\begin{subarray}{c}
n,m=0\\
n\neq m
\end{subarray}}^{N-1}\langle H_{2}\cdot(RH_{1}),R^{|n-m|-1}(H_{1}\cdot(RH_{2})-\left\langle H_{1},RH_{2}\right\rangle _{\mu})\rangle_{\mu}\underset{\lesssim N^{-1}\E\big[H_{1}^{2}(X_{0})H_{2}^{2}(X_{\tau_{1}})\big]}{-\underbrace{\frac{1}{N}\langle H_{1},RH_{2}\rangle_{\mu}^{2}}}.
\]
Using the spectral gap of the operator $R$ together with the Cauchy-Schwarz
inequality, we obtain that
\[
\left\Vert R^{|n-m|-1}(H_{1}\cdot(RH_{2})-\langle H_{1},RH_{2}\rangle_{\mu})\right\Vert _{L^{2}(\mu)}\lesssim\|H_{1}\cdot(RH_{2})\|_{L^{2}(\mu)}\L_{\gamma}^{|n-m|-1}(s_{0}).
\]
Consequently, using again Cauchy-Schwarz and the formula for the sum
of finite geometric series, we can bound the considered variance by
\begin{align*}
 & \frac{1}{N^{2}}\sum_{\begin{subarray}{c}
n,m=0\\
n\neq m
\end{subarray}}^{N-1}\|H_{2}\cdot(RH_{1})\|_{L^{2}(\mu)}\|H_{1}\cdot(RH_{2})\|_{L^{2}(\mu)}\L_{\gamma}^{|n-m|-1}(s_{0})\\
 & \qquad\lesssim\frac{1}{N}\|H_{2}\cdot(RH_{1})\|_{L^{2}(\mu)}\|H_{1}\cdot(RH_{2})\|_{L^{2}(\mu)}\\
 & \qquad\lesssim\frac{1}{N}\E\left[H_{2}^{2}(X_{0})H_{1}^{2}(X_{\tau_{1}})\right]^{1/2}\E\left[H_{1}^{2}(X_{0})H_{2}^{2}(X_{\tau_{1}})\right]^{1/2}\\
 & \qquad=\frac{1}{N}\E\left[H_{2}^{2}(X_{0})H_{1}^{2}(X_{\tau_{1}})\right].\tag*{{\qedhere}}
\end{align*}

\end{proof}
The first consequence of the previous result is the following bound
for the risk of the estimator of the invariant measure.
\begin{prop}
\label{prop:Invariant measure bound}Under Assumption~\ref{ass:drift volatility}
it holds
\begin{equation}
\E\Big[\left\Vert \mu-\hat{\mu}_{J}\right\Vert _{L^{2}}^{2}\Big]\lesssim N^{-2Js}+N^{-1}2^{J}.\label{eq:MuError}
\end{equation}
Furthermore if we choose $2^{J}\sim N^{1/(2s+3)}$ the event $\T_{0}=\{\forall x\in[0,1]\,\inf\mu/2\leq\hat{\mu}_{J}(x)\leq2\sup\mu\}$
satisfies $\P\Big(\Omega\setminus\T_{0}\Big)\lesssim N^{-\frac{2s}{2s+3}}$.\end{prop}
\begin{proof}
The explicit formula (\ref{eq:mu}) for $\mu$ shows that $\|\mu\|_{H^{s}}$
is uniformly bounded over $\Theta_{s}.$ Jackson's inequality yields
\[
\|(I-\pi_{J})\mu\|_{L^{2}}^{2}\lesssim2^{-2Js}.
\]
Using Lemma \ref{lem:Variance estimates}, we obtain 
\begin{eqnarray*}
\E\left[\|\pi_{J}\mu-\hat{\mu}_{J}\|_{L^{2}}^{2}\right] & = & \sum_{|\lambda|\leq J}\E\big[\langle\psi_{\lambda},\mu-\mu_{N}\rangle{}^{2}\big]=\sum_{|\lambda|\leq J}\Var\big[\langle\psi_{\lambda},\mu_{N}\rangle\big]\\
 & \lesssim & N^{-1}\sum_{|\lambda|\leq J}\E\left[\psi_{\lambda}^{2}(X_{0})\right]\lesssim2^{J}N^{-1}
\end{eqnarray*}
and (\ref{eq:MuError}) follows by the triangle inequality. Furthermore,
by Jackson's inequality,
\begin{eqnarray*}
\sup_{x\in[0,1]}\pi_{J}\mu(x) & \leq & \|\mu\|_{\infty}+\|(I-\pi_{J})\mu\|_{\infty}\lesssim\|\mu\|_{H^{1}}+\|(I-\pi_{J})\mu\|_{H^{1}}\lesssim1+2^{-J(s-1)}\\
\inf_{x\in[0,1]}\pi_{J}\mu(x) & \geq & \inf_{x\in[0,1]}\mu(x)-\|(I-\pi_{J})\mu\|_{\infty}\gtrsim1-2^{-J(s-1)}.
\end{eqnarray*}
Hence, for $J$ large enough, $\pi_{J}\mu$ is bounded by $\frac{3}{4}\inf\mu$
from below and $\frac{3}{2}\sup\mu$ from above. Consequently, $\hat{\mu}_{J}(x)$
lies in $[\frac{1}{2}\inf\mu,2\sup\mu]$ if $\|\hat{\mu}_{J}-\pi_{J}\mu\|_{\infty}$
is small enough. For a given constant $C>0$, Bernstein's inequality
shows
\begin{align*}
\P\Big(\|\hat{\mu}_{J}-\pi_{J}\mu\|_{\infty}>C\Big) & \leq C^{-2}\E\big[\|\pi_{J}\mu-\hat{\mu}_{J}\|_{\infty}^{2}\big]\lesssim\E\big[\|\pi_{J}\mu-\hat{\mu}_{J}\|_{H^{1}}^{2}\big]\\
 & \lesssim2^{2J}\E\Big[\|\pi_{J}\mu-\hat{\mu}_{J}\|_{L^{2}}^{2}\Big]\lesssim N^{-\frac{2s}{2s+3}}.\tag*{{\qedhere}}
\end{align*}

\end{proof}

\subsection{Analysis of the projection error}

Denote by $(\kappa_{J,i},u_{J,i})$, $i=0,1,2,...,\text{dim}V_{J}-1$,
the eigenpairs of the operator $\pi_{J}^{\mu}R\pi_{J}^{\mu}$ ordered
decreasingly with respect to the eigenvalues. Note that $(\kappa_{J,i},u_{J,i})$
are solutions of the eigenvalue problem for the operator $R$ restricted
to the finite approximation spaces $V_{J}$ on $L^{2}(\mu)$:
\begin{equation}
\langle Ru_{J,i},v\rangle_{\mu}=\kappa_{J,i}\langle u_{J,i},v\rangle_{\mu},\text{ for every }v\in V_{J}.\label{eq:eigenvalue problem for R_J}
\end{equation}
Take $u_{J,i}$ normalized in the $L^{2}$ norm. Since $\pi_{J}^{\mu}R\pi_{J}^{\mu}$
is a positive definite self-adjoint operator on $L^{2}(\mu)$ with
$\|\pi_{J}^{\mu}R\pi_{J}^{\mu}\|_{L^{2}(\mu)}\leq1$ we have $0<\kappa_{J,i}\leq1$. 
\begin{prop}
\label{prop:Bias error bounds}For sufficiently large $J$ it holds
uniformly on $\Theta_{s}$ 
\[
\left|\kappa_{J,1}-\kappa_{1}\right|+\left\Vert u_{J,1}-u_{1}\right\Vert _{H^{1}}\lesssim2^{-Js}.
\]
\end{prop}
\begin{proof}
It suffices to show that $\left|\kappa_{J,1}-\kappa_{1}\right|+\left\Vert u_{J,1}-u_{1}\right\Vert _{L^{2}}\lesssim2^{-J(s+1)}.$
Indeed, by Jackson's and Bernstein's inequalities 
\begin{eqnarray*}
\left\Vert u_{J,1}-u_{1}\right\Vert _{H^{1}} & \leq & \left\Vert u_{J,1}-\pi_{J}u_{1}\right\Vert _{H^{1}}+\left\Vert \left(I-\pi_{J}\right)u_{1}\right\Vert _{H^{1}}\lesssim2^{J}\left\Vert u_{J,1}-\pi_{J}u_{1}\right\Vert _{L^{2}}+\left\Vert \left(I-\pi_{J}\right)u_{1}\right\Vert _{H^{1}}\\
 & \lesssim & 2^{J}\left\Vert u_{J,1}-u_{1}\right\Vert _{L^{2}}+2^{J}\left\Vert (I-\pi_{J})u_{1}\right\Vert _{L^{2}}+\left\Vert \left(I-\pi_{J}\right)u_{1}\right\Vert _{H^{1}}\\
 & \lesssim & 2^{J}\left\Vert u_{J,1}-u_{1}\right\Vert _{L^{2}}+2^{-Js}
\end{eqnarray*}
where we used the upper bound (\ref{eq:eigenfunctions bounds}). 

Recall that $R$ is a compact self-adjoint positive-definite operator
on $L^{2}(\mu)$. Furthermore 
\begin{eqnarray*}
\left\Vert \left(I-\pi_{J}^{\mu}\right)u_{1}\right\Vert _{L^{2}(\mu)} & \lesssim & \left\Vert \left(I-\pi_{J}^{\mu}\right)\left(I-\pi_{J}\right)u_{1}\right\Vert _{L^{2}}\lesssim\left\Vert \left(I-\pi_{J}\right)u_{1}\right\Vert _{L^{2}}\\
 & \lesssim & 2^{-J(s+1)}\|u_{1}\|_{H^{s+1}}\lesssim2^{-J(s+1)}.
\end{eqnarray*}
Consequently, since by Lemma \ref{lem:eigenvaluesR} operator $R$
has a uniform spectral gap inequality
\[
\left\Vert \left(I-\pi_{J}^{\mu}\right)u_{1}\right\Vert _{L^{2}(\mu)}\leq\frac{\kappa_{1}-\kappa_{2}}{4\kappa_{1}}
\]
holds for $J$ large enough. It follows that we can use Theorem \ref{thm:EP for self-adjoint compact positive-definite operator}
obtaining
\[
\left|\kappa_{J,1}-\kappa_{1}\right|+\Big\|\frac{u_{J,1}}{\|u_{J,1}\|_{L^{2}(\mu)}}-\frac{u_{1}}{\|u_{1}\|_{L^{2}(\mu)}}\Big\|_{L^{2}(\mu)}\lesssim2^{-J(s+1)}.
\]
The claim follows since $\|u_{J,1}-u_{1}\|_{L^{2}}\lesssim\Big\|\frac{u_{J,1}}{\|u_{J,1}\|_{L^{2}(\mu)}}-\frac{u_{1}}{\|u_{1}\|_{L^{2}(\mu)}}\Big\|_{L^{2}(\mu)}$
by the equivalence of norms $\|\cdot\|_{L^{2}}$ and $\|\cdot\|_{L^{2}(\mu)}$.\end{proof}
\begin{cor}
\label{cor:Projected-operators-spectral gap}Projected operators $\pi_{J}^{\mu}R\pi_{J}^{\mu}$
have a uniform spectral gap, i.e. there exists $s_{1}>0$ such that
\[
\min\left\{ \left|\kappa_{J,1}\right|,\left|\kappa_{J,2}-\kappa_{J,1}\right|\right\} \geq s_{1}
\]
for every $J$ large enough.\end{cor}
\begin{proof}
Follows from the proof of Theorem \ref{thm:EP for self-adjoint compact positive-definite operator}.
\end{proof}

\subsection{Analysis of the stochastic error\label{sub:Analysis stoch error}}

Define the operator $R_{J}:V_{J}\to V_{J}$ as the restriction of
the operator $\pi_{J}^{\mu}R\pi_{J}^{\mu}$ to the finite dimensional
Hilbert space $V_{J}$. Recall that the operator $G_{J}$ was defined
by the Gram matrix of the inner product $\langle\cdot,\cdot\rangle_{\mu}$,
i.e. for $v\in V_{J}$ we have $\langle v,G_{J}v\rangle=\langle v,v\rangle_{\mu}$.
Note that by (\ref{eq:eigenvalue problem for R_J})
\begin{equation}
R_{J}u_{J,i}=\kappa_{J,i}G_{J}u_{J,i},\label{eq:Unperturbed GSEP}
\end{equation}
hence $(\kappa_{J,i},u_{J,i})$ are solutions of generalized symmetric
eigenvalue problem for $R_{J},G_{J}$. When matrix $\hat{G}_{J}$
is invertible the corresponding generalized eigenvalue problem for
$\hat{G}_{J},\hat{R}_{J}$, namely
\begin{equation}
\hat{R}_{J}\hat{u}_{J,i}=\hat{\kappa}_{J,i}\hat{G}_{J}\hat{u}_{J,i}\label{eq:Perturbed GSEP}
\end{equation}
has $\text{dim}V_{J}$ solutions that we denote by $(\hat{\kappa}_{J,i},\hat{u}_{J,i})$,
$i=0,1,...,\text{dim}V_{J}-1$. Recall that the eigenfunctions $\hat{u}_{J,i}$
are normalized in $L^{2}[0,1]$.

In this subsection we want to bound the expected error between $(\kappa_{J,1},u_{J,1})$
and $(\hat{\kappa}_{J,1},\hat{u}_{J,1})$. From the general theory
of a posteriori error bound techniques for generalized symmetric eigenvalue
problems (see Section \ref{sub:Appendix Generalized symmetric eigenvalue problems})
we know that the error between the eigenpairs can be controlled by
the norm of the residual vectors:
\[
r=\big(\hat{R}_{J}-R_{J}\big)u_{J,1}+\kappa_{J,1}\big(G_{J}-\hat{G}_{J}\big)u_{J,1}\text{ or }r^{*}=\big(R_{J}-\hat{R}_{J}\big)\hat{u}_{J,1}+\hat{\kappa}_{J,1}\big(\hat{G}_{J}-G_{J}\big)\hat{u}_{J,1}.
\]
Since the eigenpair $(\hat{\kappa}_{J,1},\hat{u}_{J,1})$ of the problem
(\ref{eq:Perturbed GSEP}) is random and depends on operators $\hat{R}_{J}$
and $\hat{G}_{J}$ it is easier to analyze the norm of the vector
$r$ rather than $r^{*}$ (cf. Lemmas \ref{lem:G hat error} and \ref{lem: Ralpha error}
where $v$ is a deterministic function). Consequently in the following
we refer to $r$ as the residual vector. In the notation of Section
\ref{sub:Appendix Generalized symmetric eigenvalue problems} we treat
the deterministic problem (\ref{eq:Unperturbed GSEP}) as a perturbed
approximation of the data dependent problem (\ref{eq:Perturbed GSEP}).
\begin{lem}
\label{lem:G hat error}For any $v\in V_{J}$ we have, uniformly on
$\Theta_{s}\times\Gamma$,
\[
\E\Big[\|(\hat{G}_{J}-G_{J})v\|_{L^{2}}^{2}\Big]\lesssim N^{-1}2^{J}\|v\|_{L^{2}}^{2}.
\]
\end{lem}
\begin{proof}
Given Lemma \ref{lem:Variance estimates}, the proof is a straight
forward estimate analogously to \citep[Lemma 4.8]{GobetHoffmannReiss:2004}.
\end{proof}
Now, we are ready to prove Lemma \ref{lem:Set T_0 size}:
\begin{proof}[Proof of Lemma \ref{lem:Set T_0 size}]
A standard Neumann series argument shows that $\hat{G}_{J}$ is invertible
on $\T_{1}$ with $\big\|\hat{G}_{J}{}^{-1}\big\|_{L^{2}}\leq2\|G_{J}^{-1}\|_{L^{2}}$.
Since the invariant density $\mu$ has a positive lower bound uniformly
on $\Theta_{s}$, for any $v\in V_{J}$ we have 
\[
\left\langle v,G_{J}v\right\rangle =\left\langle v,v\right\rangle _{\mu}=\|v\|_{L^{2}(\mu)}^{2}\gtrsim\|v\|_{L^{2}}^{2}.
\]
Hence the smallest eigenvalue of the operator $G_{J}$ is uniformly
separated from zero. This implies that $G_{J}^{-1}$ is uniformly
bounded in the operator norm. The classical Hilbert-Schmidt norm inequality
yields
\[
\big\|\hat{G}_{J}-G_{J}\big\|_{L^{2}}^{2}\leq\sum_{|\lambda|\leq J}\big\|(\hat{G}_{J}-G_{J})\psi_{\lambda}\big\|_{L^{2}}^{2}.
\]
Consequently, by Lemma \ref{lem:G hat error}, $\E\big[\big\|\hat{G}_{J}-G_{J}\big\|_{L^{2}}^{2}\big]\lesssim N^{-1}2^{2J}$
and $\P\left(\Omega\setminus\T_{1}\right)\leq N^{-1}2^{2J}$ follows
from Chebyshev's inequality.\end{proof}
\begin{lem}
\label{lem: Ralpha error}For any $v\in V_{J}$ we have, uniformly
on $\Theta_{s}\times\Gamma$,
\[
\E\Big[\|(\hat{R}_{J}-R_{J})v\|{}_{L^{2}}^{2}\Big]\lesssim N^{-1}2^{J}\|v\|_{L^{2}}^{2}.
\]
\end{lem}
\begin{proof}
By Lemma \ref{lem:Variance estimates} we obtain
\begin{alignat*}{1}
\E\Big[\|(\hat{R}_{J}-R_{J})v\|{}_{L^{2}}^{2}\Big] & =\sum_{|\lambda|\leq J}\Var\Big[\frac{1}{N}\sum_{n=0}^{N-1}\psi_{\lambda}\big(X_{\tau_{n}}\big)v\big(X_{\tau_{n}}\big)\Big]\\
 & \lesssim\sum_{|\lambda|\leq J}N^{-1}\E\left[\psi_{\lambda}^{2}\big(X_{\tau_{1}}\big)v^{2}\big(X_{0}\big)\right]\\
 & \lesssim N^{-1}\big\|\sum_{|\lambda|\leq J}\psi_{\lambda}^{2}\big\|_{\infty}\E\big[v^{2}(X_{0})\big]\\
 & \lesssim N^{-1}2^{J}\|v^{2}\|_{L^{2}(\mu)}^{2}.\tag*{{\qedhere}}
\end{alignat*}
\end{proof}
\begin{cor}
\label{Cor:Bound on residual vector}We have, uniformly on $\Theta_{s}\times\Gamma$,
the following bound on the norm of the residual vector $r=\big(\hat{R}_{J}-R_{J}\big)u_{J,1}+\kappa_{J,1}\big(G_{J}-\hat{G}_{J}\big)u_{J,1}$
\[
\E\left[\left\Vert r\right\Vert _{L^{2}}^{2}\right]\lesssim N^{-1}2^{J}.
\]
\end{cor}
\begin{proof}
Note that from Proposition \ref{prop:Bias error bounds} we know that,
for $J$ big enough, the eigenvalue $\kappa_{J,1}$ is uniformly bounded.
Consequently
\[
\E\left[\left\Vert r\right\Vert _{L^{2}}^{2}\right]\lesssim\E\left[\|(\hat{R}_{J}-R_{J})u_{1}^{J}\|_{L^{2}}^{2}\right]+\E\left[\|(\hat{G}_{J}-G_{J})u_{1}^{J}\|_{L^{2}}^{2}\right]\lesssim N^{-1}2^{J}
\]
by Lemmas \ref{lem:G hat error} and \ref{lem: Ralpha error}.\end{proof}
\begin{prop}
\label{prop:Variance error bounds}On the event $\T_{1}$ the eigenpair
$(\hat{\kappa}_{J,1},\hat{u}_{J,1})$ is the biggest nontrivial eigenpair
of the matrix $\hat{G}_{J}^{-1}\hat{R}_{J}$. Furthermore there exists
a set $\T_{2}\subset\T_{1}$ such that
\[
\P\left(\Omega\setminus\T_{2}\right)\lesssim N^{-1}2^{3J}
\]
and
\[
\E\left[\I_{\T_{2}}\cdot\left(\left|\kappa_{J,1}-\hat{\kappa}_{J,1}\right|^{2}+\left\Vert u_{J,1}-\hat{u}_{J,1}\right\Vert _{L^{2}}^{2}\right)\right]\lesssim N^{-1}2^{J}
\]
holds uniformly on $\Theta_{s}$.\end{prop}
\begin{proof}
By Theorem \ref{thm:GHEP} there exists some $0\leq i_{0}\leq\text{dim}V_{J}-1$
such that the eigenpair $(\hat{\kappa}_{J,i_{0}},\hat{u}_{J,i_{0}})$
of the problem (\ref{eq:Perturbed GSEP}) satisfies
\begin{eqnarray*}
\left|\kappa_{J,1}-\hat{\kappa}_{J,i_{0}}\right| & \leq & \big\|\hat{G}_{J}^{-1}\big\|_{L^{2}}\left\Vert r\right\Vert _{L^{2}},\\
\|u_{J,1}-\hat{u}_{J,i_{0}}\|_{L^{2}} & \leq & \frac{2\sqrt{2}}{\delta\left(\hat{\kappa}_{J,i}\right)}\big\|\hat{G}_{J}\big\|_{L^{2}}^{1/2}\big\|\hat{G}_{J}^{-1}\big\|_{L^{2}}^{3/2}\left\Vert r\right\Vert _{L^{2}},
\end{eqnarray*}
where $\delta\left(\hat{\kappa}_{J,i_{0}}\right)=\min_{j\neq i_{0}}\left\{ \left|\hat{\kappa}_{J,j}-\kappa_{J,1}\right|\right\} $
is the isolation distance of the eigenvalues $\hat{\kappa}_{J,i_{0}}$
and $\kappa_{J,1}$. Let $s_{1}$ be the uniform spectral gap of operators
$R_{J}$ (see Corollary \ref{cor:Projected-operators-spectral gap}).
Define $\T_{2}$ as the subset of $\T_{1}$ for which $i_{0}=1$ and
$\delta\left(\hat{\kappa}_{J,1}\right)\geq\frac{1}{2}s_{1}$. Since
$\big\|\hat{G}_{J}^{-1}\big\|_{L^{2}}$ and $\big\|\hat{G}_{J}\big\|_{L^{2}}$
are uniformly bounded on the event $\T_{1}$ and $\E[\left\Vert r\right\Vert _{L^{2}}^{2}]\lesssim N^{-1}2^{J}$
the desired error bound holds when we restrict to the event $\T_{2}$. 

To finish the proof we must show that $\P\left(\Omega\setminus\T_{2}\right)\lesssim N^{-1}2^{3J}$.
Denote
\[
\T_{2}=\T_{1}\cap\underset{\T_{2,1}}{\underbrace{\left\{ i_{0}=1\right\} }}\cap\underset{\T_{2,2}}{\underbrace{\{\delta(\hat{\kappa}_{J,1})\geq s_{1}/2\}}.}
\]
First, using the absolute Weyl theorem (Theorem \ref{thm:Weyl GHEP})
we observe that for any $0\leq j\leq\text{dim}V_{J}-1$
\begin{eqnarray*}
\E\left[\I_{\T_{1}}\cdot\left|\kappa_{J,j}-\hat{\kappa}_{J,j}\right|^{2}\right] & \leq & \E\Big[\I_{\T_{1}}\cdot\big\|\hat{G}_{J}^{-1}\big\|_{L^{2}}^{2}\big\|(R_{J}-\hat{R}_{J})-\kappa_{J,j}(G_{J}-\hat{G}_{J})\big\|_{L^{2}}^{2}\Big]\\
 & \lesssim & \E\Big[\I_{\T_{1}}\cdot\big\| R_{J}-\hat{R}_{J}\big\|_{L^{2}}^{2}\Big]+\kappa_{J,j}\E\Big[\I_{\T_{0}}\cdot\big\| G_{J}-\hat{G}_{J}\big\|_{L^{2}}^{2}\Big]\\
 & \lesssim & N^{-1}2^{2J}
\end{eqnarray*}
by the classical Hilbert-Schmidt norm inequality. Consequently, using
the uniform lower bound on the spectral gap of $R_{J}$, we obtain
\begin{eqnarray*}
\P\left(\T_{1}\setminus\T_{2,1}\right) & \lesssim & \E\Big[\I_{\T_{1}\setminus\T_{2,1}}\cdot\big|\kappa_{J,2}-\kappa_{J,1}\big|^{2}\Big]\\
 & \lesssim & \E\Big[\I_{\T_{1}\setminus\T_{2,1}}\cdot\big|\kappa_{J,i_{0}}-\kappa_{J,1}\big|^{2}\Big]\\
 & \lesssim & \E\Big[\I_{\T_{1}\setminus\T_{2,1}}\cdot\big|\kappa_{J,i_{0}}-\hat{\kappa}_{J,i_{0}}\big|^{2}\Big]+\E\Big[\I_{\T_{1}\setminus\T_{2,1}}\cdot\big|\hat{\kappa}_{J,i_{0}}-\kappa_{J,1}\big|^{2}\Big]\\
 & \lesssim & N^{-1}2^{2J}.
\end{eqnarray*}
Consider now the event $\T_{2,2}$. Since 
\begin{eqnarray*}
\delta\left(\hat{\kappa}_{J,1}\right) & = & \min_{j\neq1}\left|\hat{\kappa}_{J,j}-\kappa_{J,1}\right|\geq\min_{j\neq1}\left\{ \left|\kappa_{J,j}-\kappa_{J,1}\right|-\left|\hat{\kappa}_{J,j}-\kappa_{J,j}\right|\right\} \\
 & \geq & s_{1}-\max_{j\neq1}\left\{ \left|\hat{\kappa}_{J,j}-\kappa_{J,j}\right|\right\} ,
\end{eqnarray*}
we have
\begin{align*}
\P\left(\T_{1}\setminus\T_{2,2}\right) & \leq\P\big(\T_{1}\cap\big\{\max_{j\ne1}\big\{|\hat{\kappa}_{J,j}-\kappa_{J,j}|\big\}\geq s_{1}/2\big\}\big)\\
 & \leq\sum_{1<j\leq\text{dim}V_{J}-1}\P\big(\T_{1}\cap\big\{\big|\hat{\kappa}_{J,j}-\kappa_{J,j}\big|\geq s_{1}/2\big\}\big)\\
 & \lesssim\sum_{1<j\leq\text{dim}V_{J}-1}\E\Big[\I_{\T_{1}}\cdot\big|\hat{\kappa}_{J,j}-\kappa_{J,j}\big|^{2}\Big]\lesssim N^{-1}2^{3J}.\tag*{{\qedhere}}
\end{align*}

\end{proof}

\subsection{Proof of Theorem \ref{thm:Upper bounds}}

From now on we chose $2^{J}\sim N^{1/(2s+3)}.$ Recall that the biggest
negative eigenvalue of the infinitesimal generator $L$ is denoted
by $v_{1}$ which is estimated by $\hat{v}_{J,1}$ from (\ref{eq:vHat}).
\begin{lem}
\label{lem:vHat}Choose $2^{J}\sim N^{1/(2s+3)}.$ There is an event
$\mathcal{T}_{3}\subset\mathcal{T}_{2}$ satisfying $\P(\Omega\setminus\mathcal{T}_{3})\lesssim N^{-2s/(2s+3)}$
uniformly on $\Theta_{s}\times\Gamma$ and 
\[
\sup_{(\sigma,b\gamma)\in\Theta_{s}\times\Gamma}\E[\mathbf{1}_{\mathcal{T}_{3}}|v_{1}-\hat{v}_{J,1}|^{2}]\lesssim N^{-\frac{2s}{2s+3}}.
\]
In particular we can assume that $|\hat{v}_{J,1}|$ is uniformly bounded
on $\T_{3}$.\end{lem}
\begin{proof}
For convenience we denote $m:=\min I$, $M:=\max I$. On $\mathcal{T}_{2}$
we have $\hat{\kappa}_{J,1}>0$ and thus $\hat{\kappa}_{J,1}=\hat{\mathcal{L}}(-\hat{v}_{J,1})$.

\emph{Step 1:} Let us start with a consistency result for $\hat{v}_{J,1}$.
Since $\hat{\mathcal{L}}$ is non-increasing and continuous, we have
for any fixed $\epsilon\in(0,C_{1})$ with $C_{1}$ from (\ref{eq:eigenvalue bounds})
\begin{align*}
\mathbb{P}_{\gamma}(|\hat{v}_{J,1}-v_{1}|<\epsilon)\ge & \mathbb{P}_{\gamma}\big(\hat{\mathcal{L}}(-v_{1}+\epsilon)<\hat{\kappa}_{J,1}<\hat{\mathcal{L}}(-v_{1}-\epsilon)\big).
\end{align*}
Using
\begin{equation}
\delta:=\alpha me^{(v_{1}-\epsilon)M}\le\inf_{\gamma\in\Gamma}\inf_{|y+v_{1}|\le\epsilon}|\mathcal{L}'_{\gamma}(y)|,\label{eq:Lderiv}
\end{equation}
we have $|\mathcal{L}_{\gamma}(-v_{1})-\mathcal{L}_{\gamma}(-v_{1}\pm\epsilon)|\ge\delta\epsilon$
uniformly in $\gamma\in\Gamma$ and 
\begin{align*}
\mathbb{P}_{\sigma,b,\gamma}(|\hat{v}_{J,1}-v_{1}|\ge\epsilon)\le & \P\big(\kappa_{1}-\hat{\kappa}_{J,1}>\kappa_{1}-\hat{\mathcal{L}}(-v_{1}+\epsilon)\big)+\P\big(\hat{\kappa}_{J,1}-\kappa_{1}>\hat{\mathcal{L}}(-v_{1}-\epsilon)-\kappa_{1}\big)\\
\le & \sum_{y\in\{-\epsilon,+\epsilon\}}\P\big(|\hat{\kappa}_{J,1}-\kappa_{1}|+|\hat{\mathcal{L}}(-v_{1}+y)-\mathcal{L}_{\gamma}(-v_{1}+y)|>\delta\epsilon\big)\\
\le & 2\mathbb{P}_{\sigma,b}\big(|\hat{\kappa}_{J,1}-\kappa_{1}|>\frac{\delta\epsilon}{2}\big)+\sum_{y\in\{-\epsilon,+\epsilon\}}\mathbb{P}_{\gamma}\big(|\hat{\mathcal{L}}(-v_{1}+y)-\mathcal{L}_{\gamma}(-v_{1}+y)|>\frac{\delta\epsilon}{2}\big).
\end{align*}
By Propositions~\ref{prop:Bias error bounds} and \ref{prop:Variance error bounds}
and Markov's inequality the first probability is of the order $N^{-2s/(2s+3)}$
if $2^{J}\sim N^{1/(2s+3)}$. For the estimation error of $\hat{\mathcal{L}}$
Markov's inequality yields for any $y>0$
\begin{align*}
\mathbb{P}_{\gamma}\big(|\hat{\mathcal{L}}(y)-\mathcal{L}_{\gamma}(y)|>\delta\epsilon/2\big)\le & 2(\delta\epsilon)^{-2}\mathbb{E}_{\gamma}\big[|\hat{\mathcal{L}}(y)-\mathcal{L}_{\gamma}(y)|^{2}\big]\\
= & \frac{2}{N\delta^{2}\epsilon^{2}}\operatorname{Var}_{\gamma}\big(e^{-y\Delta_{1}}\big)\le\frac{2\mathcal{L}_{\gamma}(2y)}{N\delta^{2}\epsilon^{2}}.
\end{align*}
Therefore, 
\begin{equation}
\P(|\hat{v}_{J,1}-v_{1}|\ge\epsilon)\lesssim N^{-2s/(2s+3)}.\label{eq:ConsistV1}
\end{equation}

\emph{Step 2:} To determine the rate of $\hat{v}_{J,1}$, we use a
Taylor expansion which yields for some intermediate point $\xi$ between
$-v_{1}$ and $-\hat{v}_{J,1}$ 
\[
\hat{\kappa}_{J,1}=\hat{\mathcal{L}}(-\hat{v}_{J,1})=\hat{\mathcal{L}}(-v_{1})+(v_{1}-\hat{v}_{J,1})\hat{\mathcal{L}}'(\xi).
\]
Since on the other hand we have $\hat{\kappa}_{J,1}=\mathcal{L}_{\gamma}(-v_{1})+\hat{\kappa}_{J,1}-\kappa_{1}$,
we conclude
\[
v_{1}-\hat{v}_{J,1}=\frac{\mathcal{L}_{\gamma}(-v_{1})-\hat{\mathcal{L}}(-v_{1})+\hat{\kappa}_{J,1}-\kappa_{1}}{\hat{\mathcal{L}}'(\xi)},
\]
provided the denominator can be uniformly bounded with high probability.
By (\ref{eq:ConsistV1}) the event $\mathcal{T}_{3,1}:=\{|\hat{v}_{J,1}-v_{1}|<\epsilon\}$
has at least the probability $1-cN^{-2s/(2s+3)}$ for some $c>0$.
On $\mathcal{T}_{3,1}$ we have
\[
|\hat{\mathcal{L}}'(\xi)|\ge\inf_{|y+v_{1}|<\epsilon}\mathcal{L}'_{\gamma}(y)-\sup_{|y+v_{1}|<\epsilon}\big|\hat{\mathcal{L}}'(y)-\mathcal{L}'_{\gamma}(y)\big|.
\]
With $\delta$ from (\ref{eq:Lderiv}) we conclude that $|\hat{\mathcal{L}}'(\xi)|\ge\delta/2$
on the event $\mathcal{T}_{3,2}:=\{\sup_{y\in[-v_{1}-\epsilon,-v_{1}+\epsilon]}\big|\hat{\mathcal{L}}'(y)-\mathcal{L}_{\gamma}'(y)\big|^{2}<\delta/2\}$.
Note that in $\mathcal{T}_{3,2}$ we take the supremum of the empirical
processes related to $(\Delta_{n})_{n=1,\dots,N}$ acting on the function
set $\mathcal{F}:=\{[0,\infty)\ni x\mapsto xe^{-yx}:y\in[|v_{1}|-\epsilon,|v_{1}|+\epsilon]\}$.
 Since $\mathcal{F}$ is the multiplication of the identity map with
the transition class $\{e^{-yx}:y>0\}$), $\mathcal{F}$ is a Vapnik-\v{C}ervonenkis
class and admits the constant envelope function $(|v_{1}|-\epsilon)^{-1}e^{-1}$.
The empirical process theory (e.g., \citet{vanderVaartWellner1996},
Thm. 2.14.1) yields
\[
\mathbb{E}_{\gamma}\Big[\sup_{y\in[-v_{1}-\epsilon,-v_{1}+\epsilon]}\big|\hat{\mathcal{L}}'(y)-\mathcal{L}_{\gamma}'(y)\big|^{2}\Big]\lesssim\frac{1.}{N(|v_{1}|-\epsilon)^{2}}
\]
and by Markov's inequality $\mathbb{P}_{\gamma}(\Omega\setminus\mathcal{T}_{3,2})\lesssim1/N$.
With $\mathcal{T}_{3}:=\mathcal{T}_{3,1}\cap\mathcal{T}_{3,2}\cap\mathcal{T}_{2}$
we finally obtain
\begin{align*}
\E[\I_{\mathcal{T}_{3}}|v_{1}-\hat{v}_{J,1}|^{2}] & \le2\E\Big[\I_{\mathcal{T}_{3}}\frac{|\mathcal{L}_{\gamma}(-v_{1})-\hat{\mathcal{L}}(-v_{1})|^{2}+|\bar{\kappa}_{1}-\kappa_{1}|^{2}}{|\hat{\mathcal{L}}'(\xi)|^{2}}\Big]\\
 & \lesssim N^{-1}+\E\big[\I_{\mathcal{T}_{3}}|\hat{\kappa}_{J,1}-\kappa_{1}|^{2}\big]\lesssim N^{-2s/(2s+3)}.\tag*{{\qedhere}}
\end{align*}
\end{proof}
\begin{cor}
\label{cor:v=000026u}Choosing $2^{J}\sim N^{1/(2s+3)}$, there exist
an event $\T_{4}=\mathcal{T}_{0}\cap\mathcal{T}_{3}$ of high probability,
i.e. $\P\left(\Omega\setminus\T_{4}\right)\lesssim N^{-2s/(2s+3)}$,
such that the estimators $\hat{\mu}_{J}$ and $\hat{v}_{J,1}$ are
uniformly bounded on $\T_{4}$. Furthermore, for $N$ big enough,
we have uniformly on $\Theta_{s}$ and $\Gamma$
\begin{eqnarray*}
\E\Big[\I_{\T_{4}}\cdot\Big(\big|v_{1}-\hat{v}_{J,1}\big|^{2}+\big\| u_{1}-\hat{u}_{J,1}\big\|_{H^{1}}^{2}\Big)\Big] & \lesssim & N^{-2s/(2s+3)}.
\end{eqnarray*}
\end{cor}
\begin{proof}
Note that $\T_{4}$ is a subset of the events from Proposition \ref{prop:Variance error bounds},
Lemma~\ref{lem:vHat} and the event that $\hat{\mu}_{J}$ is uniformly
bounded from below and above (see Proposition \ref{prop:Invariant measure bound}).
Then $\T_{4}$ is a high probability event and by Propositions \ref{prop:Bias error bounds}
and \ref{prop:Variance error bounds}, the choice $2^{J}\sim N^{1/(2s+3)}$
yields the claimed bound of the expectation.
\end{proof}
Before we present the proof of Theorem \ref{thm:Upper bounds} we
need to another representation of the volatility estimator which allows
us to bound the derivative of the estimated eigenfunction.
\begin{lem}
\label{lem:sigmaTilde}Set $0<a<b<1.$ There exists a high probability
event $\T_{5}\subset\mathcal{T}_{4}$, $\P\left(\Omega\setminus\T_{5}\right)\lesssim N^{-2s/(2s+3)}$
such that
\[
\I_{\T_{5}}\cdot\hat{\sigma}_{J}^{2}(x)=\I_{\T_{5}}\cdot\frac{2\hat{v}_{J,1}\int_{0}^{x}\hat{u}_{J,1}(y)\hat{\mu}_{J}(y)dy}{(\hat{u}_{J,1}'(x)\vee c'_{a,b})\hat{\mu}_{J}(x)}\wedge D
\]
for a deterministic constant $c'_{a,b}>0$ satisfying $c'_{a,b}\leq c_{a,b}\leq\inf_{x\in[a,b]}u'_{1}(x).$\end{lem}
\begin{proof}
Recall that 
\[
\hat{\sigma}_{J}^{2}(x)=\frac{2\hat{v}_{J,1}\int_{0}^{x}\hat{u}_{J,1}(y)\hat{\mu}_{J}(y)dy}{\hat{u}_{J,1}'(x)\hat{\mu}_{J}(x)}\wedge D=\frac{2\hat{v}_{J,1}\int_{0}^{x}\hat{u}_{J,1}(y)\hat{\mu}_{J}(y)dy}{\hat{\mu}_{J}(x)\big(\hat{u}_{J,1}'(x)\vee\frac{2\hat{v}_{J,1}\int_{0}^{x}\hat{u}_{J,1}(y)\hat{\mu}_{J}(y)dy}{\hat{\mu}_{J}(x)D}\big)}.
\]
Let $m=\frac{1}{2}\inf\mu(x)$ and $M=2\sup\hat{\mu}_{J}$. By Proposition~\ref{prop:Invariant measure bound}
$m\le\hat{\mu}_{J}(x)\le M$ for all $x\in[0,1]$ on the event $\mathcal{T}_{0}$.
This event is especially contained in 
\[
\T_{5}:=\T_{4}\cap\Big\{4\big\|\hat{v}_{J,1}\int_{0}^{x}\hat{u}_{J,1}(y)\hat{\mu}_{J}(y)dy-v_{1}\int_{0}^{x}u_{1}(y)\mu(y)dy\big\|_{\infty}\leq d^{2}c_{a,b}m\Big\},
\]
where $\T_{4}$ is the high probability event from Corollary \ref{cor:v=000026u}.
On $\T_{5}$ it holds
\begin{eqnarray*}
\frac{2\hat{v}_{J,1}\int_{0}^{x}\hat{u}_{J,1}(y)\hat{\mu}_{J}(y)dy}{D\hat{\mu}_{J}(x)} & \geq & \frac{2v_{1}\int_{0}^{x}u_{1}(y)\mu(y)dy-2|\hat{v}_{J,1}\int_{0}^{x}\hat{u}_{J,1}(y)\hat{\mu}_{J}(y)dy-v_{1}\int_{0}^{x}u_{1}(y)\mu(y)dy|}{D\hat{\mu}_{J}(x)}\\
 & = & \frac{\sigma^{2}(x)u_{1}'(x)\mu(x)-2|\hat{v}_{J,1}\int_{0}^{x}\hat{u}_{J,1}(y)\hat{\mu}_{J}(y)dy-v_{1}\int_{0}^{x}u_{1}(y)\mu(y)dy|}{D\hat{\mu}_{J}(x)}\\
 & \geq & \frac{d^{2}c_{a,b}m}{2MD}=:c'_{a,b}.
\end{eqnarray*}
Furthermore, by Corollary \ref{cor:v=000026u}, using Markov and triangle
inequalities, it is easy to check that $\P(\Omega\setminus\T_{5})\lesssim N^{-\frac{2s}{2s+3}}$,
cf. estimate (\ref{eq:E1}) below.
\end{proof}

\begin{proof}[Proof for the volatility estimator]
 Set $0<a<b<1.$ Note first that since $\P\left(\Omega\setminus\T_{5}\right)\lesssim N^{-\frac{2s}{2s+3}}$
and $\sigma,\hat{\sigma}$ are bounded we just have to verify that
$\E[\I_{\T_{5}}\cdot\|\sigma^{2}-\hat{\sigma}^{2}\|_{L^{2}([a,b])}^{2}]\lesssim N^{-\frac{2s}{2s+3}}.$
Denote $\tilde{u}_{J,1}'(x)=\hat{u}_{J,1}'(x)\vee c'_{a,b}$ and $\tilde{\sigma}_{J}^{2}(x)=\frac{2\hat{v}_{J,1}\int_{0}^{x}\hat{u}_{J,1}(y)\hat{\mu}_{J}(y)dy}{\tilde{u}_{J,1}'(x)\hat{\mu}_{J}(x)}$.
Since for $x\in[a,b]$ the functions $u_{1}'$ and $\mu$ are uniformly
separated from zero, we have that on $\T_{5}$ 
\begin{align*}
\big|\sigma^{2}(x)-\hat{\sigma}_{J}^{2}(x)\big| & \leq\Big|\frac{2v_{1}\int_{0}^{x}u_{1}(y)\mu(y)dy}{u_{1}'(x)\mu(x)}-\frac{2\hat{v}_{J,1}\int_{0}^{x}\hat{u}_{J,1}(y)\hat{\mu}_{J}(y)dy}{\tilde{u}_{J,1}'(x)\hat{\mu}_{J}(x)}\Big|\\
 & =\Big|\frac{2\left(v_{1}\int_{0}^{x}u_{1}(y)\mu(y)dy-\hat{v}_{J,1}\int_{0}^{x}\hat{u}_{J,1}(y)\hat{\mu}_{J}(y)dy\right)}{u_{1}'(x)\mu(x)}-\frac{\tilde{\sigma}_{J}^{2}(x)\left(u_{1}'(x)\mu(x)-\tilde{u}_{J,1}'(x)\hat{\mu}_{J}(x)\right)}{u'_{1}(x)\mu(x)}\Big|\\
 & \lesssim\left|v_{1}\int_{0}^{x}u_{1}(y)\mu(y)dy-\hat{v}_{J,1}\int_{0}^{x}\hat{u}_{J,1}(y)\hat{\mu}_{J}(y)dy\right|+|\tilde{\sigma}_{J}^{2}(x)|\Big|\frac{u_{1}'(x)\mu(x)-\tilde{u}_{J,1}'(x)\hat{\mu}_{J}(x)}{u_{1}'(x)}\Big|\\
 & =:A_{1}(x)+A_{2}(x).
\end{align*}
Observe that since $\hat{\mu}_{J}$ is uniformly bounded on the event
$\T_{5}$ and since the eigenfunction $\hat{u}_{1}$ is normalized
the Cauchy-Schwarz inequality grants that $\int_{0}^{x}\hat{u}_{J,1}(y)\hat{\mu}_{J}(y)dy$
is uniformly bounded. Hence, 
\begin{align}
A_{1}(x) & =\big|v_{1}\big(\int_{0}^{x}u_{1}(y)\mu(y)dy-\int_{0}^{x}\hat{u}_{J,1}(y)\hat{\mu}_{J}(y)dy\big)+\int_{0}^{x}\hat{u}_{J,1}(y)\hat{\mu}_{J}(y)dy(v_{1}-\hat{v}_{J,1})\big|\nonumber \\
 & \lesssim\big|\int_{0}^{x}u_{1}(y)\mu(y)dy-\int_{0}^{x}\hat{u}_{J,1}(y)\hat{\mu}_{J}(y)dy\big|+|v_{1}-\hat{v}_{J,1}|\nonumber \\
 & \lesssim\big|\int_{0}^{x}u_{1}(y)(\mu(y)-\hat{\mu}_{J,1}(y))dy\big|+\big|\int_{0}^{x}(u_{1}(y)-\hat{u}_{J,1}(y))\hat{\mu}_{J}(y)dy\big|+|v_{1}-\hat{v}_{1,J}|\nonumber \\
 & \leq\|u_{1}\|_{L^{2}}\left\Vert \mu-\hat{\mu}_{J}\right\Vert _{L^{2}}+\|u_{1}-\hat{u}_{J,1}\|_{L^{2}}\big\|\hat{\mu}_{J}\big\|_{L^{2}}+|v_{1}-\hat{v}_{1,J}|\nonumber \\
 & =\left\Vert \mu-\hat{\mu}_{J}\right\Vert _{L^{2}}+\|u_{1}-\hat{u}_{J,1}\|_{L^{2}}+|v_{1}-\hat{v}_{J,1}|.\label{eq:E1}
\end{align}
Furthermore, since $\tilde{\sigma}_{J}^{2}(x)$ is uniformly bounded
on $\T_{5}$ 
\begin{align}
A_{2}(x) & \lesssim|\mu(x)-\hat{\mu}_{J}(x)|+\frac{|\hat{\mu}_{J}(x)|}{|u_{1}'(x)|}|u_{1}'(x)-\tilde{u}_{J,1}'(x)|\nonumber \\
 & \lesssim|\mu(x)-\hat{\mu}_{J}(x)|+|u_{1}'(x)-\tilde{u}_{J,1}'(x)|\nonumber \\
 & \lesssim|\mu(x)-\hat{\mu}_{J}(x)|+|u_{1}'(x)-\hat{u}_{J,1}'(x)|.\label{eq:E2}
\end{align}
Consequently,
\begin{align*}
\E\Big[\I_{\T_{5}}\cdot\big\|\sigma^{2}-\hat{\sigma}_{J}^{2}\big\|_{L^{2}}^{2}\Big] & \lesssim\E\left[\I_{\T_{5}}\cdot\big(\|A_{1}\|_{L^{2}}^{2}+\|A_{2}\|_{L^{2}}^{2}\big)\right]\\
 & \lesssim\E\left[\I_{\T_{5}}\cdot\left(\left\Vert \mu-\hat{\mu}_{J}\right\Vert _{L^{2}}^{2}+\|u_{1}-\hat{u}_{J,1}\|_{H^{1}}^{2}+\left|v_{1}-\hat{v}_{J,1}\right|^{2}\right)\right]\\
 & \lesssim N^{-2s/(2s+3)}.\tag*{{\qedhere}}
\end{align*}

\end{proof}

\begin{proof}[Proof for the drift estimator]
 To obtain the upper bound on the drift term first note that using
Bernstein's inequality we can extend the proofs of Propositions \ref{prop:Bias error bounds}
and \ref{prop:Variance error bounds} to obtain
\begin{equation}
\E\Big[\I_{\T_{4}}\cdot\left\Vert u_{1}-\hat{u}\right\Vert _{H^{2}}^{2}\Big]\lesssim N^{-\frac{2(s-1)}{2s+3}}.\label{eq:ef H2 error}
\end{equation}
Let $\T_{6}=\T_{5}\cap\{\inf_{x\in[a,b]}\hat{u}'_{J,1}(x)\geq c_{a,b}/2\}\cap\{\|\hat{u}_{J,1}\|_{H^{2}}\leq2\|u_{1}\|_{H^{2}}\}$.
By Lemma~\ref{lem:sigmaTilde} and (\ref{eq:ef H2 error}) we obtain
that $\P(\Omega\setminus\T_{6})\lesssim N^{-\frac{2(s-1)}{2s+1}}$.
Since both $b$ and $\hat{b}$ are bounded in $L^{2}$, we can restrict
the error analysis to the high probability event $\T_{6}$. Recall
the definition of $\tilde{b}$ from (\ref{eq:drift est definition}).
Since $\|b\|_{L^{2}([a,b])}\leq D$ we have $\|\hat{b}_{J}-b\|_{L^{2}([a,b])}\leq\|\tilde{b}_{J}-b\|_{L^{2}([a,b])}$.
Consequently, it remains to show 
\[
\E\Big[\I_{\T_{6}}\cdot\|\tilde{b}_{J}-b\|_{L^{2}([a,b])}\Big]\lesssim N^{-\frac{2(s-1)}{2s+3}}.
\]
 On $\T_{6}$, for $x\in[a,b]$ we have
\begin{align*}
|\tilde{b}_{J}(x)-b(x)|\leq & \Big|\frac{\hat{v}_{J,1}\hat{u}_{J,1}(x)}{\hat{u}_{J,1}'(x)}-\frac{\tilde{\sigma}_{J}^{2}(x)\hat{u}_{J,1}''(x)}{2\hat{u}_{J,1}'(x)}-\frac{v_{1}u_{1}(x)}{u_{1}'(x)}+\frac{\sigma^{2}(x)u_{1}''(x)}{2u_{1}'(x)}\Big|\\
\le & |u_{1}'(x)|^{-1}\Big|\hat{v}_{J,1}\hat{u}_{J,1}(x)-v_{1}u_{1}(x)+\frac{\sigma^{2}(x)}{2}u_{1}''(x)-\frac{\tilde{\sigma}_{J}^{2}(x)}{2}\hat{u}_{J,1}''(x)\Big|\\
 & +\frac{|\tilde{b}_{J}(x)|}{|u_{1}'(x)|}\big|u_{1}'(x)-\hat{u}_{J,1}'(x)\big|.
\end{align*}
The uniform lower bound on $|u_{1}'|$ yields
\begin{align*}
\|\tilde{b}_{J}-b\|_{L^{2}([a,b])}^{2}\lesssim & \|\hat{v}_{J,1}\hat{u}_{J,1}-v_{1}u_{1}\|_{L^{2}([a,b])}^{2}+\|\tilde{\sigma}_{J}^{2}\hat{u}_{J,1}''-\sigma^{2}u_{1}''\|_{L^{2}([a,b])}^{2}\\
 & \quad+\|\tilde{b}_{J}\|_{L^{2}([a,b])}^{2}\|\hat{u}_{J,1}'-u_{1}'\|_{L^{\infty}([a,b])}^{2}\\
=: & B_{1}+B_{2}+B_{3}.
\end{align*}
We will estimate these three terms separately. Corollary~\ref{cor:v=000026u}
and the normalization of $\hat{u}_{J,1}$ yield
\[
\E[\mathbf{1}_{\mathcal{T}_{6}}B_{1}]\le\E\big[\mathbf{1}_{\mathcal{T}}\big(|\hat{v}_{J,1}-v_{1}|^{2}\|\hat{u}_{J,1}\|_{L^{2}}^{2}+|v_{1}|^{2}\|\hat{u}_{J,1}-u_{1}\|_{L^{2}}^{2}\big)\big]\lesssim N^{-2s/(2s+3)}.
\]
The second term can be decomposed into
\[
B_{2}\le2\|\tilde{\sigma}_{J}^{2}-\sigma^{2}\|_{\infty}^{2}\|u_{1}''\|_{L^{2}}^{2}+2\|\tilde{\sigma}_{J}^{2}\|_{\infty}^{2}\|\hat{u}_{J,1}''-u_{1}''\|_{L^{2}}^{2}.
\]
From (\ref{eq:E1}) and (\ref{eq:E2}) we can easily verify that
\[
\|\hat{\sigma}_{J}^{2}-\sigma^{2}\|_{\infty}\lesssim|\hat{v}_{J,1}-v_{1}|+\|\hat{u}_{J,1}-u_{1}\|_{H^{2}}+\|\hat{\mu}_{J}-\mu\|_{H^{1}}.
\]
Since $\hat{\sigma}_{J}^{2}$ is bounded by construction, we conclude
\[
\E[\mathbf{1}_{\mathcal{T}_{6}}B_{2}]\le\E\big[\mathbf{1}_{\mathcal{T}_{6}}\big(|\hat{v}_{J,1}-v_{1}|^{2}+\|\hat{u}_{J,1}-u_{1}\|_{H^{2}}^{2}++\|\hat{\mu}_{J}-\mu\|_{H^{1}}^{2}\big)\big]\lesssim N^{-2(s-1)/(2s+3)}.
\]
For the last term it holds
\[
\E[\mathbf{1}_{\mathcal{T}_{6}}B_{3}]\le\E\big[\mathbf{1}_{\mathcal{T}_{6}}\|\tilde{b}_{J}\|_{L^{2}([a,b])}^{2}\|\hat{u}_{J,1}-u_{1}\|_{H^{2}}^{2}\big]\lesssim N^{-2(s-1)/(2s+3)}
\]
since $\|\tilde{b}_{J}\|_{L^{2}[(a,b)]}$ is uniformly bounded on
$\T_{6}$.
\end{proof}

\section{Proof of the lower bounds\label{sec:PrLowBound}}

First note that estimating the sampling distribution $\gamma$ has
no impact on the convergence rates, because the Laplace transform
can be estimated with the parametric rate. Therefore, it suffices
to use the same distribution $\gamma\in\Gamma$ for all alternatives.
Throughout this section we thus fix some $\gamma\in\Gamma$ which
admits a bounded Lebesgue density on $[0,T]$ for some $T>0$. 

Without loss of generality we can suppose that $(1,0)\in\Theta_{s}.$
To construct the alternatives, let $\psi$ be a compactly supported
wavelet in $H^{s}$ with one vanishing moment. We set $\psi_{jk}(x)=2^{j/2}\psi(2^{j}x-k)$
and denote by $K_{j}\subset\Z$ a maximal set of indices $k$ such
that $\text{supp}(\psi_{jk})\subset[a,b]$ and $\text{supp}(\psi_{jk})\cap\text{supp}(\psi_{jk'})=\emptyset$
holds for all $k,k'\in K_{j}$, $k\neq k'.$ For a constant $\delta>0$
and all $\epsilon=(\epsilon_{k})\in\left\{ -1,1\right\} ^{|K_{j}|}$
we define
\[
S_{\epsilon}(x)=S_{\epsilon}(j,x)=\Big(2+\delta\sum_{k\in K_{j}}\epsilon_{k}\psi_{jk}(x)\Big)^{-1}.
\]
Choosing $\delta\sim2^{-j(s+1/2)}$ yields $(\sqrt{2S_{\epsilon}},S_{\epsilon}')\in\Theta_{s}$.
The corresponding diffusions $X^{(\epsilon)}$ are defined by their
generators
\begin{eqnarray*}
L_{\epsilon}f(x) & = & S_{\epsilon}(x)f''(x)+S'_{\epsilon}(x)f'(x),\\
\text{dom}(L_{\epsilon}) & = & \text{dom}(L).
\end{eqnarray*}
Note that for any $\epsilon$ the invariant measure of $X^{(\epsilon)}$
is given by Lebesgue measure on $[0,1]$. For $\epsilon,\epsilon'$
with $\|\epsilon-\epsilon'\|_{\ell^{2}}=2$ we have
\[
S_{\epsilon'}(x)-S_{\epsilon}(x)=\pm2\delta\psi_{jk}(x)S_{\epsilon'}(x)S_{\epsilon}(x).
\]
Since $S_{\epsilon},S_{\epsilon'}$ converge uniformly to $1/2$ as
$j\to\infty$, the $L^{2}$-distances of the volatility functions
and the drift functions of the alternatives $\epsilon$ and $\epsilon'$
are bounded by
\[
\|2S_{\epsilon'}-2S_{\epsilon}\|_{L^{2}}\gtrsim\delta,\quad\|S'_{\epsilon'}-S'_{\epsilon}\|_{L^{2}}\gtrsim2^{j}\delta.
\]
Therefore, Assouad's lemma and $\delta\sim2^{-j(s+1/2)}$ yield for
all estimators $\bar{\sigma}^{2}$ and $\bar{b}$
\begin{align}
\sup_{(\sigma,b)\in\Theta_{s}}\E\Big[\|\bar{\sigma}^{2}-\sigma^{2}\|_{L^{2}([a,b])}^{2}\Big] & \gtrsim2^{j}\delta=2^{-2sj},\nonumber \\
\sup_{(\sigma,b)\in\Theta_{s}}\E\Big[\|\bar{b}-b\|_{L^{2}([a,b])}^{2}\Big] & \gtrsim2^{3j}\delta=2^{-2(s+1)j},\label{eq:lowerBound}
\end{align}
provided the Kullback-Leibler divergence between the distributions
of $(X_{\tau_{n}}^{(\epsilon)})_{n=0,\dots,N}$ and $(X_{\tau_{n}}^{(\epsilon')})_{n=0,\dots,N}$
remains uniformly bounded for all alternatives $\epsilon,\epsilon'$
with $\|\epsilon-\epsilon'\|_{\ell^{2}}=2$.

To bound the Kullback-Leibler divergence, we have to take into account
the random observation times. Denote the transition density of $(X_{t})_{t\ge0}$
by $p_{t}(x,y)dy=\mathbb{P}_{\sigma,b}(X_{t}=dy|X_{0}=x)$ for $x,y\in[0,1],t\ge0$.
By the independence of the observation time $\tau$ and the process
$X$ we have 
\begin{align*}
Rf(x) & =\E\left[f(X_{\tau})|X_{0}=x\right]=\int_{0}^{\infty}P_{t}f(x)\gamma(dt)=\int_{0}^{\infty}\int_{0}^{1}p_{t}(x,y)f(y)dy\gamma(dt).
\end{align*}
For one dimensional diffusions with bounded drift and differentiable
volatility, which is uniformly separated from zero, we know that

\[
p_{t}(x,y)\leq c_{0}\Big(1+\frac{1}{\sqrt{t}}\Big)
\]
with $c_{0}>0$ depending only on the bounds for the drift and volatility
(see \citet[Thm. 1]{QianZheng:2002}). The assumption $\mathbb{E}[\tau^{-1/2}]<\infty$
thus ensures that 
\[
r(x,y)=\int_{0}^{\infty}p_{t}(x,y)\gamma(dt)
\]
 is a well defined kernel of operator $R.$ We obtain the following
generalization of Proposition 6.4 in \citep{GobetHoffmannReiss:2004}:
\begin{lem}
\label{lem:ConvergenceOfProbabilities}Assume $\mathbb{E}_{\gamma}\left[\tau^{-1/2}\right]<\infty.$
If $(\sigma_{n},b_{n})\in\Theta_{s}$, $n\geq0$, such that 
\[
\lim_{n\to\infty}\left\Vert \sigma_{n}-\sigma_{0}\right\Vert _{\infty}=0\quad\text{and}\quad\lim_{n\to\infty}\left\Vert b_{n}-b_{0}\right\Vert _{\infty}=0,
\]
then the corresponding kernels $r^{(n)}(x,y)dy=\mathbb{P}_{\sigma_{n},b_{n}}(X_{\tau}\in dy|X_{0}=x)$
satisfy
\[
\lim_{n\to\infty}\big\| r^{(n)}-r^{(0)}\big\|_{\infty}=0.
\]

\end{lem}
Note that the bounded Lebesgue density $\gamma$ near the origin specially
ensures that $\mathbb{E}_{\gamma}[\tau^{-1/2}]<\infty$. 
\begin{proof}
Due to the bound $\|p_{t}^{(n)}(\cdot,\cdot)\|_{\infty}\lesssim1+t^{-1/2}$,
dominated convergence yields
\begin{eqnarray*}
\big\| r^{(n)}-r^{(0)}\big\|_{\infty} & = & \sup_{x,y\in[0,1]}\Big|\int_{0}^{\infty}\left(p_{t}^{(n)}(x,y)-p_{t}^{(0)}(x,y)\right)\gamma(dt)\Big|\\
 & \leq & \int_{0}^{\infty}\big\| p_{t}^{(n)}-p_{t}^{(0)}\big\|_{\infty}\gamma(dt).
\end{eqnarray*}
By \citep[Prop. 6.4]{GobetHoffmannReiss:2004} this tends to zero.
\end{proof}
Exactly as in \citep[Sect. 5.2]{GobetHoffmannReiss:2004}, this lemma
allows us to bound the Kullback-Leibler divergence by $N\|r_{\epsilon'}-r_{\epsilon}\|_{L^{2}([0,1]^{2})}^{2}$
for kernels $r_{\epsilon'}$ and $r_{\epsilon}$ of $R_{\epsilon'}$
and $R_{\epsilon}$, respectively, for any $\epsilon,\epsilon'$ with
$\|\epsilon-\epsilon'\|_{\ell^{2}}=2$. Note that $\left\Vert r_{\epsilon'}-r_{\epsilon}\right\Vert _{L^{2}([0,1]^{2})}$
is the Hilbert-Schmidt norm distance $\|R-R^{\epsilon'}\|_{HS}=\|(R^{\epsilon}-R^{\epsilon'})|_{V}\|_{HS}$
where 
\[
V=\Big\{ f\in L^{2}([0,1])\Big|\int_{0}^{1}f=0\Big\}.
\]
We will bound the Hilbert-Schmidt norm by the difference of the inverses
of the generators, which are, in contrast to the generators itself,
bounded operators. Recall that $R=\L(-L)$ for the Laplace transform
$\L(z)=\int_{0}^{\infty}e^{-tz}\gamma(dt),z\ge0$. By the functional
calculus for operators the function $f(z)=\L\left(-z^{-1}\right)$
maps $(L_{\epsilon}|_{V})^{-1}$ to $R^{\epsilon}|_{V}$. Furthermore,
$f$ is uniformly Lipschitz on $(-\infty,0)$:
\begin{lem}
Suppose that $\gamma\in\Gamma$ admits a bounded Lebesgue density
on $[0,T]$ for some $T>0$. Then we have 

\[
c:=\sup_{z<0}\Big|\frac{1}{z^{2}}\int_{0}^{\infty}te^{t/z}\gamma(dt)\Big|<\infty.
\]
\end{lem}
\begin{proof}
We decompose 
\[
\sup_{z<0}\Big|\frac{1}{z^{2}}\int_{0}^{\infty}te^{t/z}\gamma(dt)\Big|\le\sup_{z<0}\Big|\frac{1}{z^{2}}\int_{0}^{T}te^{t/z}\gamma(dt)\Big|+\sup_{z<0}\Big|\frac{1}{z^{2}}\int_{T}^{\infty}te^{t/z}\gamma(dt)\Big|=:S_{1}+S_{2}.
\]
Due to the bounded Lebesgue density on $[0,T]$, we estimate the first
term by substituting $s=t/z$
\[
S_{1}\lesssim\sup_{z<0}z^{-2}\int_{0}^{T}te^{t/z}dt=\sup_{z<0}\int_{T/z}^{0}se^{s}ds=\int_{-\infty}^{0}se^{s}ds<\infty.
\]
For the second term note that the function $g_{a}(x)=x^{2}e^{-ax}$
takes maximum at $x=2/a$ and $g\left(2/a\right)=4a^{-2}e^{-2}.$
Consequently,

\[
S_{2}\le\sup_{z<0}\int_{T}^{\infty}tg_{t}(|z|^{-1})\gamma(dt)=\int_{T}^{\infty}\frac{4}{te^{2}}\gamma(dt)\leq\frac{4}{Te^{2}}<\infty.\tag*{{\qedhere}}
\]

\end{proof}
We conclude 
\[
\big\| r_{\epsilon'}-r_{\epsilon}\big\|_{L^{2}([0,1]^{2})}=\big\|(R^{\epsilon}-R^{\epsilon'})\big|{}_{V}\big\|_{HS}\leq c\big\|(L_{\epsilon}|_{V})^{-1}-(L_{\epsilon'}|_{V})^{-1}\big\|_{HS}\lesssim\delta2^{-j}=2^{-j(2s+3)/2},
\]
by the estimate for the difference of inverses of the generators that
was established in \citep[Sect. 5.3]{GobetHoffmannReiss:2004}. In
order to bound $N\|r_{\epsilon'}-r_{\epsilon}\|_{L^{2}([0,1]^{2})}^{2}$,
we thus choose $j$ such that $2^{j}\sim N^{1/(2s+3)}$. In view of
(\ref{eq:lowerBound}) we have proven Theorem~\ref{thm:Lower bounds}.\qed

\section{Proof for the adaptive estimator\label{sec:proofAdaptiv}}

In order to show that Lepski's method works, we need the following
concentration result. It slightly generalizes the corresponding concentration
inequalities by \citet[Theorems 10 and 11]{NicklSohl:2015} for a
low-frequently observed reflected diffusion to random sampling times.
\begin{prop}
\label{prop:concentration}Grant Assumptions~\ref{ass:times} and
\ref{ass:drift volatility} with $s>5/2$ and $\gamma\in\Gamma$,
$\mathbb{E}_{\gamma}[\tau^{-1/2}]\le D.$ There is a constant $c>0$
depending only on $d,D,I$ and $\alpha,$ such that, for any $\kappa>0,N\in\N$
and any $f\in L^{2}(\R)\cap L^{\infty}(\R),g\in L^{2}(\R^{2})\cap L^{\infty}(\R^{2})$:
\begin{align*}
\P\Big(\Big|\sum_{n=0}^{N}\big(f(X_{\tau_{n}})-\E[f(X_{0})]\big)\Big|>\kappa\Big)\lesssim & \exp\Big(-c\min\Big\{\frac{\kappa^{2}}{N\|f\|_{L^{2}}^{2}},\frac{\kappa}{(\log N)\|f\|_{\infty}}\Big\}\Big)
\end{align*}
and
\begin{align*}
\P\Big(\Big|\sum_{n=0}^{N-1}\big(g(X_{\tau_{n}},X_{\tau_{n+1}})- & \E[g(X_{0},X_{\tau_{1}})]\big)\Big|>\kappa\Big)\\
\lesssim & \exp\Big(-c\min\Big\{\frac{\kappa^{2}}{N\|g\|_{L^{2}}^{2}},\frac{\kappa}{(\log N)\|g\|_{\infty}}\Big\}\Big).
\end{align*}
\end{prop}
\begin{proof}
The conditions of the Markov chain concentration result by \citet[Theorem 6]{Adamczak:2007}
have to be verified. This can be done along the lines of the proofs
in \citep{NicklSohl:2015} using Lemma~\ref{lem:eigenvaluesR} and
noting that the transition density of the time-changed chain $(X_{\tau_{n}})_{n\ge1}$
is given by $p_{\gamma}(x,y)=\int_{0}^{\infty}p_{t}(x,y)\gamma(dt)$
where $p_{t}(x,y)$ denotes the transition density of the diffusion
$(X_{t})_{t\ge0}$. The condition $s>5/2$ ensures that the transition
density $p_{\gamma}$ is bounded from below uniformly on $[0,1]^{2}$.
Indeed, $p_{\gamma}(x,y)\geq K\gamma(I)\geq K\alpha,$ where $K$
is the uniform lower bound on $\inf_{t\in I}p_{t}$ obtained in \citep[Proposition 9]{NicklSohl:2015}.
Since $\|p_{t}\|_{\infty}\lesssim1+t^{-1/2},$ the condition $\mathbb{E}_{\gamma}[\tau^{-1/2}]<\infty$
ensures a uniform upper bound on $p_{\gamma}.$
\end{proof}
To analyze the performance of $\tilde{\sigma}^{2}$, we first decompose
its estimation error into a deterministic and a stochastic error term.
In what follows, $C=C(d,D,I,\alpha)$ denotes a numeric constant which
may vary from line to line. We deduce from the proof of Theorem~\ref{thm:Upper bounds}
on the there defined event $\mathcal{T}_{5},$ that for any $J\in\mathcal{J}_{N}$
\begin{align}
\|\hat{\sigma}_{J}^{2}-\sigma^{2}\|_{L^{2}}\le & C\big(\|\mu-\hat{\mu}_{J}\|_{L^{2}}+\|u_{1}-\hat{u}_{J,1}\|_{H^{1}}+|v_{1}-\hat{v}_{J,1}|\big)\nonumber \\
\le & C\big(\|\mu-\hat{\mu}_{J}\|_{L^{2}}+\|u_{1}-\hat{u}_{J,1}\|_{H^{1}}+|\kappa_{1}-\hat{\kappa}_{J,1}|+|\mathcal{L}_{\gamma}(-v_{1})-\hat{\mathcal{L}}_{\gamma}(-v_{1})|\big)\nonumber \\
\le & D_{J}+S_{J},\label{eq:errorDecomp}
\end{align}
where 
\begin{align*}
D_{J}:= & C\big(\|(I-\pi_{J})\mu\|_{L^{2}}+\|u_{1}-u_{J,1}\|_{H^{1}}+|\kappa_{1}-\kappa_{J,1}|\big),\\
S_{J}:= & C\big(\|\pi_{J}\mu-\hat{\mu}_{J}\|_{L^{2}}+\|u_{J,1}-\hat{u}_{J,1}\|_{H^{1}}+|\kappa_{J,1}-\hat{\kappa}_{J,1}|+|\mathcal{L}_{\gamma}(-v_{1})-\hat{\mathcal{L}}_{\gamma}(-v_{1})|\big).
\end{align*}
Due to the smoothness of the invariant measure, Jackson's inequality
and Proposition~\ref{prop:Bias error bounds}, there is some $\beta>0$,
depending on $\psi,d$ and $D$ such that 
\[
D_{J}\le\beta2^{-Js}.
\]
We need that $S_{J}$ concentrates around zero. Recalling the definition
of the residual vector
\[
r=\big(\hat{R}_{J}-R_{J}\big)u_{J,1}+\kappa_{J,1}\big(G_{J}-\hat{G}_{J}\big)u_{J,1},
\]
Bernstein's inequality and Theorem~\ref{thm:GHEP} on generalized
symmetric eigenvalue problems yield, on the event $\mathcal{T}_{2}$
from Proposition~\ref{prop:Variance error bounds}, that
\[
\|u_{J,1}-\hat{u}_{J,1}\|_{H^{1}}+|\kappa_{J,1}-\hat{\kappa}_{J,1}|\le C2^{J}\|u_{J,1}-\hat{u}_{J,1}\|_{L^{2}}+|\kappa_{J,1}-\hat{\kappa}_{J,1}|\leq\|r\|_{L^{2}}\big(C2^{J}+1\big).
\]

\begin{cor}
\label{cor:concentration}Under the conditions of Proposition~\ref{prop:concentration},
for any $\tau>1$ there exist $\eta_{1},\eta_{2},\eta_{3}>1$, such
that, for all $J$ with $2^{J}\lesssim\frac{N}{(\log N)^{2}\log\log N}$,
we have 
\begin{align}
\P\Big(\|\pi_{J}\mu-\hat{\mu}_{J}\|_{L^{2}}>2^{\frac{J}{2}}\eta_{1}\sqrt{\frac{\log\log N}{N}}\Big) & \lesssim(\log N)^{-\tau},\label{eq:concMu}\\
\P\Big(\|r\|_{L^{2}}>2^{\frac{J}{2}}\eta_{2}\sqrt{\frac{\log\log N}{N}}\Big) & \lesssim(\log N)^{-\tau},\label{eq:concR}\\
\P\Big(|\mathcal{L}_{\gamma}(-v_{1})-\hat{\mathcal{L}}_{\gamma}(-v_{1})|>\eta_{3}\sqrt{\frac{\log\log N}{N}}\Big) & \lesssim(\log N)^{-\tau}.\label{eq:concL}
\end{align}
In particular, there is a $\Lambda>0$ such that $\P(4S_{J}>s_{J})\lesssim(\log N)^{-\tau}$
for $s_{J}=s_{J}(\Lambda)$ from (\ref{eq:radius}).\end{cor}
\begin{proof}
Fix $\tau>1$. Since $\|\psi_{\lambda}\|_{\infty}\lesssim2^{|\lambda|/2}$,
for $|\lambda|\leq J$, using Proposition~\ref{prop:concentration}
we obtain 
\begin{align*}
\P\Big(|\langle\psi_{\lambda},\mu-\mu_{N}\rangle|>\eta_{1}\sqrt{\frac{\log\log N}{N}}\Big) & \lesssim\exp\Big(-c\min\Big\{\frac{\eta_{1}^{2}N(\log\log N)}{N\|\psi_{\lambda}\|_{L^{2}}^{2}},\frac{\eta_{1}\sqrt{N(\log\log N)}}{(\log N)\|\psi_{\lambda}\|_{\infty}}\Big\}\Big)\\
 & \lesssim\exp\Big(-c\eta_{1}\min\Big\{\log\log N,\frac{\sqrt{N(\log\log N)}}{(\log N)2^{J/2}}\Big\}\Big)\\
 & \lesssim(\log N)^{-c\eta_{1}}\lesssim(\log N)^{-\tau},
\end{align*}
for some $\eta_{1}$ big enough. Applying a usual chaining argument,
this concentration inequality carries over to $\max_{|\lambda|\le J}|\langle\psi_{\lambda},\mu-\mu_{N}\rangle|$,
cf. \citep[Theorem 2.1]{baraud:2010} and \citep[Theorem 12]{NicklSohl:2015}.
Since $\|\mu_{J}-\hat{\mu}_{J}\|_{L^{2}}^{2}=\sum_{|\lambda|\le J}|\langle\psi_{\lambda},\mu-\mu_{N}\rangle|^{2},$
it follows that
\[
\P\Big(\|\pi_{J}\mu-\hat{\mu}_{J}\|_{L^{2}}^{2}>\eta_{1}^{2}2^{J}\frac{\log\log N}{N}\Big)\lesssim\P\Big(\max_{|\lambda|\leq J}|\langle\psi_{\lambda},\mu-\mu_{N}\rangle|^{2}>\eta_{1}^{2}\frac{\log\log N}{N}\Big)\lesssim(\log N)^{-\tau}.
\]

To prove (\ref{eq:concR}), note first that since $|\kappa_{J,1}|\leq1,$
we have 
\[
\|r\|_{L^{2}}\leq\|\big(\hat{R}_{J}-R_{J}\big)u_{J,1}\|_{L^{2}}+\|\big(G_{J}-\hat{G}_{J}\big)u_{J,1}\|_{L^{2}}.
\]
By Proposition \ref{prop:Bias error bounds} $\|u_{J,1}\|_{L^{2}},\|u_{J,1}\|_{\infty}\lesssim1$
holds for $J$ big enough. Using the second inequality in Proposition~\ref{prop:concentration},
we obtain
\begin{align*}
 & \P\Big(|\langle\psi_{\lambda},\big(\hat{R}_{J}-R_{J}\big)u_{J,1}\rangle|>\eta_{2}\sqrt{\frac{\log\log N}{N}}\Big)\\
 & \qquad\qquad\lesssim\exp\Big(-c\eta_{2}\min\Big\{\frac{N(\log\log N)}{N},\frac{\sqrt{N(\log\log N)}}{(\log N)2^{J/2}}\Big\}\Big)\lesssim(\log N)^{-C\eta_{2}}\lesssim(\log N)^{-\tau},
\end{align*}
for $\eta_{2}$ big enough. Since $\|\big(\hat{R}_{J}-R_{J}\big)u_{J,1}\|_{L^{2}}=\sum_{|\lambda|\leq J}|\langle\psi_{\lambda},\big(\hat{R}_{J}-R_{J}\big)u_{J,1}\rangle|^{2}$,
we conclude again that 
\[
\P\Big(\|\big(\hat{R}_{J}-R_{J}\big)u_{J,1}\|_{L^{2}}>\eta_{2}2^{\frac{J}{2}}\sqrt{\frac{\log\log N}{N}}\Big)\lesssim(\log N)^{-\tau}.
\]
Arguing similarly we deduce also $\P\Big(\|\big(G_{J}-\hat{G}_{J}\big)u_{J,1}\|_{L^{2}}>\eta_{2}2^{\frac{J}{2}}\sqrt{\frac{\log\log N}{N}}\Big)\lesssim(\log N)^{-\tau}$
and thus (\ref{eq:concR}) holds.

The concentration inequality~(\ref{eq:concL}) follows from the classical
Bernstein inequality. Indeed, we have
\[
\hat{\mathcal{L}}_{\gamma}(-v_{1})-\mathcal{L}_{\gamma}(-v_{1})=\frac{1}{N}\sum_{n=1}^{N}\xi_{n}\quad\mbox{with}\quad\xi_{n}:=e^{v_{1}\Delta_{n}}-\mathbb{E}_{\gamma}[e^{v_{1}\Delta_{n}}],
\]
where, by Assumption~\ref{ass:times} the random variables $\xi_{n}$
are independent, centered and deterministically bounded by $2$ (because
$v_{1}<0$). Since $\operatorname{Var}_{\gamma}(\xi_{n})\le\mathcal{L}_{\gamma}(-2v_{1})\le1$,
we can choose $\eta_{3}$ uniformly for all $\gamma\in\Gamma$.
\end{proof}
We can now prove the convergence rate for the adaptive estimator.
\begin{proof}[Proof of Theorem~\ref{thm:adaptiveEstimation}]
Let us introduce the oracle projection level
\[
J^{*}:=\min\big\{ J\in\mathcal{J}_{N}:\beta2^{-Js}<s_{J}/4\big\}.
\]
By the choice of $\mathcal{J}_{N}$ we deduce $2^{J^{*}}\sim(N/\log\log N)^{1/(2s+3)}$
and $s_{J^{*}}^{2}\sim(\log\log N/N)^{2s/(2s+3)}$. Since the number
of elements in $\mathcal{J}_{N}$ is of order $\log N$, Proposition~\ref{prop:concentration}
yields $\P(\mathcal{A}_{N})\to1$ for the event
\[
\mathcal{A}_{N}:=\big\{\forall J\in\mathcal{J}_{N}:4S_{J}\le s_{J}\big\}\cap\mathcal{T}_{6}
\]
with $\mathcal{T}_{6}$ from the proof of Theorem~\ref{thm:Upper bounds}.
Due to the decomposition~(\ref{eq:errorDecomp}), on $\mathcal{A}_{N}$
we have for every$J\in\mathcal{J}_{N}$:
\[
\|\hat{\sigma}_{J}^{2}-\sigma^{2}\|_{L^{2}}\le D_{J}+S_{J}\le\beta2^{-Js}+s_{J}.
\]
Hence, for all $J\ge J^{*},$ $J\in\mathcal{J}_{N}$, we obtain
\[
\|\hat{\sigma}_{J}^{2}-\sigma^{2}\|_{L^{2}[a,b]}\le\frac{1}{2}s_{J},
\]
and thus, by the triangle inequality,
\[
\|\hat{\sigma}_{J}^{2}-\hat{\sigma}_{J^{*}}^{2}\|_{L^{2}[a,b]}\le s_{J},
\]
 for all $J\ge J^{*}$, $J\in\mathcal{J}_{N}$. By definition of $\hat{J}$,
we conclude that $\hat{J}\le J^{*}$ on the event $\mathcal{A}_{N}$.
We conclude that
\[
\|\tilde{\sigma}^{2}-\sigma^{2}\|_{L^{2}[a,b]}\le\|\hat{\sigma}_{\hat{J}}^{2}-\hat{\sigma}_{J^{*}}^{2}\|_{L^{2}[a,b]}+\|\hat{\sigma}_{J^{*}}^{2}-\sigma^{2}\|_{L^{2}[a,b]}\leq s_{J^{*}}+\frac{1}{2}s_{J^{*}}\le\frac{3}{2}s_{J^{*}}.\tag*{{\qedhere}}
\]
 
\end{proof}
\appendix

\section{Stability of the eigenvalue problems}

\subsection{Compact, self-adjoint, positive-definite operators}
\begin{thm}
\label{thm:EP for self-adjoint compact positive-definite operator}Consider
$T$ a compact, self-adjoint and positive-definite operator on some
Hilbert space $\mathcal{H}=\left(H,\|\cdot\|\right)$. Denote its
eigenpairs by $\left(\lambda_{i},x_{i}\right)_{i=1,2,...}$, normalized
so that $\|x_{i}\|=1$ and ordered decreasingly with respect to the
eigenvalues. Let $V\subset H$ be a finite dimensional subspace of
$H$, and $\pi$ the orthogonal projection on $V$. Assume that the
biggest eigenvalue $\lambda_{1}$ is simple and that
\[
\left\Vert \left(I-\pi\right)x_{1}\right\Vert <\frac{\lambda_{1}-\lambda_{2}}{6\lambda_{1}}.
\]
Consider the projected operator $\pi T\pi$ and denote its normalized,
ordered decreasingly, eigenpairs by $\left(\lambda_{i}^{V},x_{i}^{V}\right)_{i=1,2,...,\text{{dim\}(\ensuremath{V_{J}})}}}$
. Then

\[
\left|\lambda_{1}-\lambda_{1}^{V}\right|+\left\Vert x_{1}-x_{1}^{V}\right\Vert \leq C\left\Vert \left(I-\pi\right)x_{1}\right\Vert 
\]
holds, where the constant $C$ depends only on the size of the spectral
gap $\lambda_{1}-\lambda_{2}$ and the first eigenvalue $\lambda_{1}$.\end{thm}
\begin{proof}
Since $T$ is self-adjoint and positive-definite $\|T\|=\sup_{x\in H}\frac{\langle Tx,x\rangle}{\|x\|^{2}}=\lambda_{1}$.
By the variational characterization of the eigenvalues 
\begin{equation}
\lambda_{i}^{V}=\sup_{\begin{subarray}{c}
S\subset V\\
\text{dim}(S)=i
\end{subarray}}\inf_{y\in S}\frac{\langle y,Ty\rangle}{\|y\|^{2}}\leq\sup_{\begin{subarray}{c}
S\subset H\\
\text{dim}(S)=i
\end{subarray}}\inf_{y\in S}\frac{\langle y,Ty\rangle}{\|y\|^{2}}=\lambda_{i}.\label{eq:eigenvalues order}
\end{equation}
Furthermore
\begin{eqnarray*}
\lambda_{1}-\lambda_{1}^{V} & \leq & \frac{\left\langle \left(\lambda_{1}-\pi T\pi\right)\left(\pi x_{1}\right),\pi x_{1}\right\rangle }{\left\Vert \pi x_{1}\right\Vert ^{2}}=\frac{\left\langle \pi T\left(I-\pi\right)x_{1},\pi x_{1}\right\rangle }{\left\Vert \pi x_{1}\right\Vert ^{2}}\\
 & \leq & \frac{\left\Vert \pi T\left(I-\pi\right)x_{1}\right\Vert }{\left\Vert \pi x_{1}\right\Vert }\leq\|T\|\frac{\left\Vert \left(I-\pi\right)x_{1}\right\Vert }{\left\Vert \pi x_{1}\right\Vert }\\
 & \leq & \|T\|\frac{\left\Vert \left(I-\pi\right)x_{1}\right\Vert }{1-\left\Vert \left(I-\pi\right)x_{1}\right\Vert }.
\end{eqnarray*}
Since $\left|\lambda_{1}-\lambda_{1}^{V}\right|\leq2\|T\|$, from
the inequality $\frac{z}{1-z}\land2\leq3z$ for $z=\left\Vert \left(I-\pi\right)x_{1}\right\Vert $
follows that 

\[
\left|\lambda_{1}-\lambda_{1}^{V}\right|\leq3\|T\|\left\Vert \left(I-\pi\right)x_{1}\right\Vert .
\]

Since by (\ref{eq:eigenvalues order}) holds $\lambda_{2}^{V}\leq\lambda_{2}$
and $\|T\|\left\Vert \left(I-\pi\right)x_{1}\right\Vert <\frac{\lambda_{1}-\lambda_{2}}{6}$
we have
\begin{eqnarray*}
\left|\lambda_{1}^{V}-\lambda_{2}^{V}\right| & \geq & \lambda_{1}^{V}-\lambda_{2}=\left|\lambda_{1}-\lambda_{2}\right|-\left|\lambda_{1}-\lambda_{1}^{V}\right|\\
 & \geq & \lambda_{1}-\lambda_{2}-3\|T\|\left\Vert \left(I-\pi\right)x_{1}\right\Vert \geq\frac{1}{2}\left(\lambda_{1}-\lambda_{2}\right).
\end{eqnarray*}
Consequently the projected operator $\pi T\pi$ has a spectral gap
of size $\rho\geq\frac{\lambda_{1}-\lambda_{2}}{2}$ and in particular
the eigenvalue $\lambda_{1}^{V}$ is simple. Define the residual vector
$r=\left(\pi T\pi-T\right)x_{1}.$ Then
\begin{eqnarray*}
\|r\|=\left\Vert \left(\pi T\pi-T\right)x_{1}\right\Vert  & \leq & \left\Vert \pi T\pi x_{1}-\pi Tx_{1}\right\Vert +\lambda_{1}\left\Vert \pi x_{1}-x_{1}\right\Vert \\
 & \leq & \left(\|T\|+\lambda_{1}\right)\left\Vert \left(I-\pi\right)x_{1}\right\Vert .
\end{eqnarray*}
Consequently, in order to prove $\left\Vert x_{1}-x_{1}^{V}\right\Vert \leq C\left\Vert \left(I-\pi\right)x_{1}\right\Vert $,
it suffices to justify that
\[
\left\Vert x_{1}-x_{1}^{V}\right\Vert \leq\frac{3\rho^{2}}{2\sqrt{2}}\left\Vert r\right\Vert 
\]
Let $P$ be the spectral projection on the eigenspace of operator
$\pi T\pi$ corresponding to the eigenvalue $\lambda_{1}^{V}.$ Let
$R\left(\pi T\pi,z\right)=(\pi T\pi-z)^{-1}$ be the resolvent operator.
Using Cauchy's integral representation of the spectral projection
(see Lemma 6.4 from \citep{Chatelin:1983}) and $|\lambda_{1}-\lambda_{1}^{V}|\leq\rho$
we find 
\begin{eqnarray*}
\left\Vert x_{1}-Px_{1}\right\Vert  & = & \frac{1}{2\pi}\Big\|\varoint_{S(\lambda_{1},3\rho/2)}\frac{R\left(\pi T\pi,z\right)}{\lambda_{1}-z}dz\left(\pi T\pi-T\right)x_{1}\Big\|\\
 & \leq & \frac{3\rho}{2}\|r\|\sup_{z\in S(\lambda_{1},3\rho/2)}\left\Vert R\left(\pi T\pi,z\right)\right\Vert .
\end{eqnarray*}
Since operator $\pi T\pi$ is self adjoint on $\mathcal{H}$ we know
that (see Proposition 2.32 from \citep{Chatelin:1983}) $\left\Vert R\left(\pi T\pi,z\right)\right\Vert =\left(\text{dist}\left(z,\sigma\left(\pi T\pi\right)\right)\right)^{-1}$.
Consequently
\begin{eqnarray*}
\sup_{z\in S(\lambda_{1},3\rho/2)}\left\Vert R\left(\pi T\pi,z\right)\right\Vert  & = & \sup_{z\in S(\lambda_{1},3\rho/2)}\left(\text{dist}\left(z,\sigma\left(\pi T\pi\right)\right)\right)^{-1}\leq\frac{\rho}{2}.
\end{eqnarray*}
It remains to bound the distance between the eigenvectors. Since $x_{1}$
and $x_{1}^{V}$ are normalized
\begin{eqnarray*}
\left\Vert x_{1}^{V}-x_{1}\right\Vert ^{2} & = & 2-2\langle x_{1}^{V},x_{1}\rangle\leq2-2\langle x_{1}^{V},x_{1}\rangle^{2}\\
 & = & 2\left(1+\langle x_{1}^{V},x_{1}\rangle\right)\left(1-\langle x_{1}^{V},x_{1}\rangle\right)=2\left\Vert x_{1}-\langle x_{1}^{V},x_{1}\rangle x_{1}^{V}\right\Vert ^{2}.
\end{eqnarray*}
Since $\lambda_{1}^{V}$ is simple, the right hand side is equal to
$2\left\Vert x_{1}-Px_{1}\right\Vert ^{2}$.
\end{proof}

\subsection{Generalized symmetric eigenvalue problems.\label{sub:Appendix Generalized symmetric eigenvalue problems}}

In this section we want to sketch the a posteriori technique of solving
generalized symmetric eigenvalue problems (GSEP). GSEPs have been
studied extensively in chapter VI of \citep{StewartSun:MatrixPerturbationTh}.
For the error analysis in the case of standard matrix eigenvalue problems
we refer to Chapter 1 of \citep{Chatelin:1983} or Chapter V of \citep{StewartSun:MatrixPerturbationTh}.
A particularly useful reference for various eigenvalue problems is
\citep{SIAM:Templates:GHEP}.

Consider $A,B\in\R^{n\times n}$ real, symmetric matrices with $B$
positive definite. We call a pair $(\lambda,x)\in\R\times\left(\R^{n}\setminus\left\{ 0\right\} \right)$
an eigenpair of the generalized symmetric eigenvalue problem (GSEP)
for matrices $A,B$ if 
\begin{equation}
Ax=\lambda Bx.\label{eq:ApGSEP}
\end{equation}
Furthermore we adapt the notation of the standard eigenvalue problems
calling $\lambda$ the eigenvalue and $x$ the eigenvector. An eigenpair
is normalized if $\|x\|=1$, where $\|x\|=\big(\sum_{i=1}^{n}x_{i}^{2}\big)^{\frac{1}{2}}$
is the Euclidean norm on $\R^{n}$.

Using Cholesky decomposition of matrix $B=DD^{*}$ one can reduce
the generalized problem (\ref{eq:ApGSEP}) to the standard eigenvalue
problem for matrix $D^{-1}AD^{-*}$. We deduce that problem (\ref{eq:ApGSEP})
has $n$ solutions $(\lambda_{i},x_{i})_{i=1,..,n}$, all eigenvalues
are real and we can ordered the eigenpairs with respect to the eigenvalues
$\lambda_{1}\geq\lambda_{2}\geq...\geq\lambda_{n}$. Furthermore corresponding
eigenvectors $(x_{i})_{i=1,..,n}$ form a $B-$orthogonal basis of
$\R^{n}.$

Consider now perturbed matrices $\tilde{A}$, $\tilde{B}$ with $\tilde{B}$
positive definite and the corresponding GSEP: 
\begin{equation}
\tilde{A}\tilde{x}=\tilde{\lambda}\tilde{B}\tilde{x}.\label{eq:ApPGSEP}
\end{equation}
We want to formulate error bounds between $(\tilde{\lambda}_{1},\tilde{x}_{1})$
and $(\lambda_{1},x_{1})$. To that purpose form the residual vector
\[
r=A\tilde{x}_{1}-\tilde{\lambda}_{1}B\tilde{x}_{1}=(A-\tilde{A})\tilde{x}_{1}+\tilde{\lambda}_{1}(\tilde{B}-B)\tilde{x}_{1}.
\]
The standard a posteriori procedure is to find a matrix $E=E(\tilde{\lambda}_{1},\tilde{x}_{1})$
such that 
\begin{eqnarray}
(A+E)\tilde{x}_{1} & = & \tilde{\lambda}_{1}B\tilde{x}_{1},\label{eq:perturbed with B}\\
\|E\| & = & \|r\|.\nonumber 
\end{eqnarray}
Since we replaced in (\ref{eq:perturbed with B}) the perturbed matrix
$\tilde{B}$ by $B$, the final step is to reduce (\ref{eq:perturbed with B})
and (\ref{eq:ApGSEP}) to the standard eigenvalue problems using the
Cholesky decomposition of $B$. Then we can apply the standard error
bounds expressed in terms of the perturbation matrix $E$. We obtain
\begin{thm}
\label{thm:GHEP} There exists a normalized eigenpair $(\lambda_{i},x_{i})$,
$1\leq i\leq n$ such that
\begin{align*}
|\lambda_{i}-\tilde{\lambda}_{1}| & \leq\left\Vert B^{-1}\right\Vert \|r\|,\\
\|x_{i}-\tilde{x}_{1}\| & \leq\frac{2\sqrt{2\kappa(B)}}{\delta(\lambda_{i})}\left\Vert B^{-1}\right\Vert \|r\|.
\end{align*}
where $\kappa(B)=\|B\|\|B^{-1}\|$is the condition number of matrix
$B$ and $\delta(\lambda_{i})$ is the so called localizing distance,
i.e. $\delta(\lambda_{i})=\min_{j\neq i}\left|\lambda_{j}-\tilde{\lambda}_{1}\right|$.
\end{thm}
The disadvantage of the above procedure is that we obtain an existence
result that gives no information how the eigenpair $(\lambda_{i},x_{i})$
is related to $(\lambda_{1},x_{1})$. This is a typical downside for
a posteriori methods that are supposed to provide information how
far the calculated solution is from the nearest exact solution but
are not intended to compare ordered eigenpairs. A helpful result is
the absolute Weyl theorem for generalized hermitian definite matrix
pairs, established by Y. Nakatsukasa \citep{Nakatsukasa:2010}. For
readers convenience we state below the theorem in the form presented
in \citep[Theorem 8.3]{Nakatsukasa:2011}.
\begin{thm}
\label{thm:Weyl GHEP} Let $\lambda_{1}\geq...\geq\lambda_{n}$ and
$\tilde{\lambda}_{1}\geq...\geq\tilde{\lambda}_{n}$ be respectively
exact and approximated eigenvalues of problems (\ref{eq:ApGSEP})
and (\ref{eq:ApPGSEP}). Denote $\Delta A=A-\tilde{A}$ and $\Delta B=B-\tilde{B}$.
Then
\begin{eqnarray*}
\left|\lambda_{i}-\tilde{\lambda}_{i}\right| & \leq & \left\Vert \tilde{B}^{-1}\right\Vert \left\Vert \Delta A-\lambda_{i}\Delta B\right\Vert ,\\
\left|\lambda_{i}-\tilde{\lambda}_{i}\right| & \leq & \left\Vert B^{-1}\right\Vert \left\Vert \Delta A-\tilde{\lambda}_{i}\Delta B\right\Vert ,
\end{eqnarray*}
for all $i=1,...,n.$
\end{thm}
\bibliographystyle{apalike}
\bibliography{bibliography}

\end{document}